\documentclass[11pt,a4paper]{amsart}

\usepackage{pdfsync} 
\usepackage{amssymb, amsmath, amsthm, amsfonts,mathtools, latexsym,bm}
\usepackage[colorlinks = true,
            linkcolor = blue,
            urlcolor  = blue,
            citecolor = purple,
            anchorcolor = blue]{hyperref}

\usepackage{tikz-cd}
\usepackage{dsfont}
\usepackage{extarrows}
\usepackage{mathrsfs}
\usepackage{arydshln}
\usepackage{enumerate}
\usepackage{booktabs}
\usepackage[square, comma, sort&compress, numbers]{natbib}
\usepackage{float}
\usepackage{hyperref}

\usepackage[figuresright]{rotating}

\newcommand{\tabincell}[2]{\begin{tabular}{@{}#1@{}}#2\end{tabular}}

\numberwithin{equation}{section}
\theoremstyle{plain}
\newtheorem{theorem}{Theorem}[section]
\newtheorem{proposition}[theorem]{Proposition}
\newtheorem{corollary}[theorem]{Corollary}
\newtheorem{lemma}[theorem]{Lemma}

\theoremstyle{definition}

\theoremstyle{remark}

\newcommand{\m}[2]{\ensuremath{M_{2^{#1}}^{#2}}}
\newcommand{\mm}[2]{\ensuremath{M_{p^{#1}}^{#2}}}
\newcommand{\Z}{\ensuremath{\mathbb{Z}}}
\newcommand{\zz}[1]{\ensuremath{\mathbb{Z}/2^{#1}}}
\newcommand{\zp}[1]{\ensuremath{\mathbb{Z}/p^{#1}}}
\newcommand{\CC}{\ensuremath{C^{n+2,t}_{r}}}
\newcommand{\ul}{\ensuremath{\underline}}
\newcommand{\ol}{\ensuremath{\overline}}
\newcommand{\cc}{\ensuremath{C^{n+2,r}_{r}}}
\newcommand{\uc}{\ensuremath{\ul{C}}}
\newcommand{\oc}{\ensuremath{\ol{C}}}
\newcommand{\ceta}{\ensuremath{C^{n+2}_\eta}}

\newcommand{\A}[2]{\ensuremath{\mathsf{A}_{#1}^{#2}}}

\newcommand{\SW}{Spanier-Whitehead\,}

\newcommand{\C}{\ensuremath{\bm{C}}}
\newcommand{\E}{\ensuremath{\mathcal{E}}}
\newcommand{\xra}{\ensuremath{\xrightarrow}}
\newcommand{\an}{\ensuremath{\mathsf{A}_n^2}}

\DeclareMathOperator{\coker}{coker}
\DeclareMathOperator{\im}{im}

\DeclareMathOperator{\sq}{Sq}
\DeclareMathOperator{\Hom}{Hom}
\DeclareMathOperator{\Aut}{Aut}

\newcommand{\mat}[4]{\ensuremath{\footnotesize{\begin{pmatrix}
    #1&#2\\
    #3&#4
  \end{pmatrix}}}}

\newcommand{\ma}[2]{\ensuremath{\left(\begin{smallmatrix}
  2^{#1}&\eta\\
  0~&2^{#2}
\end{smallmatrix}}\right)}

\newenvironment{amssidewaystable}
  {\begin{sidewaystable}\vspace*{.5\textwidth}\begin{minipage}{\textheight}\centering}
  {\end{minipage}\end{sidewaystable}}

\title[Homotopy Classes of Based Maps between $\mathsf{A}_n^2$-complexes]{The Homotopy Classification of Based Maps between $\mathsf{A}_n^2$-complexes}
\author{Pengcheng Li}
\email{lipcaty@outlook.com}
\address{Department of Mathematics, School of Sciences, Great Bay University, Dongguan, Guangdong \rm{523000}, China}
\email{lipcaty@outlook.com}

\keywords{$\mathsf{A}_n^2$-complexes, Chang complexes, sets of homotopy classes, self-homotopy equivalences}
\subjclass[2020]{55Q05, 55P10}

\begin{document}

\maketitle

\begin{abstract}
  Let $X,Y$ be $(n-1)$-connected finite pointed CW-complexes of dimension at most $n+2$, $n\geq 3$. In this paper we give elementary proofs of the abelian group structure of $[X,Y]$ of homotopy classes of based maps from $X$ to $Y$, which was due to Baues and Schmidt. Furthermore, we determine the explicit generators associated to $[X,Y]$. As an application, we compute certain (sub)groups of self-homotopy equivalences of certain Chang complexes. 
\end{abstract}


\section{Introduction}\label{sec:intro}
By  \emph{$\A{n}{k}$-complexes} we mean $(n-1)$-connected finite CW-complexes of dimension at most $n+k$, $n\geq k+1$. We say that an $\A{n}{k}$-complex $X$ is \emph{elementary} or \emph{indecomposable} if $X$ does not admit a nontrivial wedge decomposition; otherwise it is \emph{decomposable}. To avoid confusion we shall subsequently use the word ``indecomposable'' instead of ``elementary" to describe $\an$-complexes.  It is well-known that indecomposable $\A{n}{1}$-complexes consist of spheres $S^{n},S^{n+1}$ and \emph{indecomposable Moore spaces} $\mm{r}{n}=S^{n}\cup_{p^r} e^{n+1}$ of homotopy type $(\Z/p^r,n)$, where $p$ is a prime, $r$ is a positive integer. Motivated by Whitehead's work \cite{Whitehead48,Whitehead49}, in 1950 Chang \cite{Chang1950} proved that every $\an$-complex is homotopy equivalent to a finite wedge sum of suitable suspensions of indecomposable $\A{n}{1}$-complexes and the following four \emph{indecomposable Chang complexes}:
\begin{align*}
  &C^{n+2,t}=(S^n\vee S^{n+1})\cup_{\binom{\eta}{2^t}}\C S^{n+1},\quad C^{n+2}_r=S^n\cup_{(2^r,\eta)}\C (S^n\vee S^{n+1}),\\[1ex]
  &\ceta=S^n\cup_\eta\C S^{n+1},\quad \CC=(S^n\vee S^{n+1})\cup_{\ma{r}{t}}\C (S^n\vee S^{n+1}).
\end{align*}
Here $\bm C X$ denotes the reduced cone on the space $X$, the  matrices enclosed within round brackets serve as representations for the attaching maps; $\eta$ is the iterated suspensions of the Hopf map $\eta\colon S^3\to S^2$ (without confusion we simply denote $\Sigma^{n-2}\eta $ by $\eta$ for different $n$); $r,t$ are positive integers. 

For based CW-complexes $X,Y$, let $[X,Y]$ be the set of homotopy classes of based maps from $X$ to $Y$.  In 1950s, Hilton \cite{Hilton50,Hilton51} computed the homotopy groups $\pi_{n+i}(X)$ of $\an$-complexes $X$ for $i\leq 2$, while Brown and Copeland \cite{BC59} determined the groups $[X,Y]$ of indecomposable Moore spaces of dimension at most $n+2$. As indicated in \cite{Hilton53}, the suspension map $\Sigma \colon [X,Y]\to [\Sigma X,\Sigma Y]$ is an isomorphism for $n\geq 3$ when $X,Y$ are \emph{indecomposable} $\an$-complexes, while it is an isomorphism for $n\geq 4$ when $X,Y$ are general $\an$-complexes, by the generalized Freudenthal suspension theorem (cf. \cite[Theoerm 1.21]{Cohenbook}).  In 1985, Baues \cite{Baues85} exhibited a complete list of the group structure of $[X,Y]$ for indecomposable $\an$-complexes $X,Y$, as worked out by T. Schmidt \cite{Schmidt84}. Under the supervision of Baues, Schmidt obtained the group structure of $[X,Y]$ by analysing the following two short exact sequences of abelian groups:
\begin{align}
  E(X,Y)\rightarrowtail &{[X,Y]}\twoheadrightarrow \mathrm{PH}_n^2(X,Y),\label{SES:BS84:1}\\[1ex]
  \Gamma^{(n)}(X,Y)\rightarrowtail&{[X,Y]}\twoheadrightarrow {\mathrm{Prin}(X,Y)/\simeq}.\label{SES:BS84:2}
\end{align}
See  \cite{Schmidt84} for detailed constructions. Schmidt computed the groups in (\ref{SES:BS84:1}) and (\ref{SES:BS84:2}) on both sides of $[X,Y]$ and proved that there are exactly three cases for the groups $[X,Y]$ of indecomposable $\an$-complexes $X,Y$:
\begin{enumerate}[(i)]
  \item For some $\an$-complexes $X,Y$,  $E(X,Y)=0$, hence $[X,Y]\cong \mathrm{PH}_n^2(X,Y)$;
  \item For some $\an$-complexes $X,Y$,  $\Gamma^{(n)}(X,Y)=0$, and hence we have $[X,Y]\cong \mathrm{Prin}(X,Y)/\simeq$;
  \item  For those indecomposable $\an$-complexes $X,Y$ such that the extension (\ref{SES:BS84:1}) is a real short exact sequence,  the extensions (\ref{SES:BS84:2}) are splitting, and hence there are isomorphisms 
  \[ [X,Y]\cong \Gamma^{(n)}(X,Y)\oplus\mathrm{Prin}(X,Y)/\simeq.\]   
\end{enumerate}

Although Schmidt constructed partial generators of the groups $[X,Y]$,  \cite{Schmidt84} doesn't include all the generators of the groups $[X,Y]$, particularly in cases where $X$ and $Y$ are indecomposable Chang complexes of three or four cells. The main purpose of this paper is to determine complete generating sets  of the groups $[X,Y]$ listed in \cite{Baues85,Schmidt84}, where $X,Y$ are indecomposable $\an$-complexes $X,Y$, $n\geq 3$. We shall obtain the groups $[X,Y]$ and their associated generators by elementary and direct approach, which differs from that of Baues and Schmidt. As shown in Section \ref{sec:proofs-tables}, some of the groups $[X,Y]$ can be easily computed by computing the (short) exact sequences for $[X,Y]$ induced by certain homotopy cofibre sequence for $X$ or certain homotopy fibre sequence for $Y$. For the groups $[X,Y]$ that cannot be computed by the above ``one step'', we shall relate certain exact sequence for $[X,Y]$ to that for $[X,Z]$ or that for $[W,Y]$, where $Z,W$ are suitable $\an$-complexes; then the groups $[X,Y]$ and the associated generators can be obtained from that of $[X,Z]$ or that of $[W,Y]$. 
The characterizations of complete generators of the groups $[X,Y]$ of indecomposable $\an$-complexes $X,Y$ are summarized in Section \ref{sec:tables}. 

There are scores of situations in which the group structure of $[X,Y]$ with explicit generators play an important role. Firstly, it is very natural to apply these group structure and explicit generators to study self-homotopy equivalences of $\an$-complexes, which are the main applications discussed in the paper. Secondly, explicit generators can be used to study homotopy decomposition of loop spaces of certain Chang complexes by determining  suitable homotopy cofibre sequences for the smash product $\cc\wedge \cc$ \cite{ZLP}. Thirdly, the $(2n+1)$-skeleton of an $(n-1)$-connected $(2n+2)$-dimensional manifold is an $\an$-complex, the group structure of $[X,Y]$ and their explicit generators are vital to determine the homotopy type of the suspension of the manifold, which has many applications in geometry and physics, see \cite{ST19,lipc23,CS22}. Finally, the group structure of $[X,Y]$ with $X,Y$ indecomposable Moore spaces and their generators are cornerstones of Oka's work \cite{Oka84} on the ring spectrum structure of the Moore spectrum $\Sigma^{\infty}(S^1\cup_q\C S^1)$. We believe that our generators of $[X,Y]$ would be valuable in the characterization of multiplicative structure on the suspension spectra of certain Chang complexes.

As an application of the groups of $[X,Y]$, together with their explicit generators, we study self-homotopy equivalences of $\an$-complexes $X,Y$, $n\geq 4$.
Let $\E(X)$ be the group of the homotopy classes of based self-homotopy equivalences of a pointed space $X$. For indecomposable $\an$-complexes $X$ with $n\geq 3$, most of the groups $\E(X)$ were known, see \cite[Part IV]{Baues85}. As an immediate consequence of results in Section \ref{sec:tables},  we have the following complete characterization of $\E(\CC)$ (Theorem \ref{thm:ECC}). Note that only the splitness of the short exact sequence (\ref{ES:ECC}) in the case $t=r$ is new, since other information in Theorem \ref{thm:ECC} was firstly obtained by Schmidt, see \cite{Baues85,Schmidt84}.

\begin{theorem}\label{thm:ECC}
  Let $n\geq 3$ and let $r,t\geq 1$ be integers.
 There is a split short exact sequence
  \begin{equation}
    \begin{tikzcd}[column sep=small]
    \zz{\min(r,t)+1}\ar[r,tail]&\E(\CC)\ar[r,two heads,"\phi"]&\Aut(\zz{t+1})\oplus\Aut(\zz{r+1}),
  \end{tikzcd}\tag{$\natural$}\label{ES:ECC}
  \end{equation}
  where $\min(r,t)$ denotes the minimum of $r$ and $t$, $\phi(f)=(\pi_{n+1}(f),\pi^{n+1}(f))$.  The associated action is given by \begin{gather*}
    \Aut(\zz{t+1})\oplus\Aut(\zz{r+1})\to \Aut(\zz{min(r,t)+1}),\\
    (\phi,\varphi)\mapsto p_\ast(\phi)q_\ast(\varphi^{-1}),
  \end{gather*}
  where $p\colon \zz{t+1}\to\zz{\min(r,t)+1}$ and $q\colon \zz{r+1}\to \zz{\min(r,t)+1}$ are the canonical projections.
\end{theorem}

Let $\E_\sharp^k(X)$ and $\E_\ast(X)$ be the subgroups of $\E(X)$ consisting of based homotopy equivalences that induce the identity on the first $k$ homotopy groups and all integral homology groups, respectively. The subgroups $\E_\sharp^{k}(X)$ and $\E_\ast(X)$ of Moore spaces $X$ have been computed in several papers, such as \cite{AM98,Jeong10,CL14}.
Given suspended spaces $X,Y$, the subset
\[\mathcal{Z}_\sharp^k(X,Y)\coloneqq\{f\in [X,Y]| f_\sharp=0\colon \pi_i(X)\to \pi_i(Y),~~i\leq k\}\]
 is a subgroup of $[X,Y]$ under addition for each $k\geq 0$.
Combining former work on subgroups of self-homotopy equivalences \cite{ALM01,Pavesic10}, we obtain the following general formulas.
\begin{theorem}\label{thm:LDU}
  Let $X_1,\cdots,X_m$ be $\an$-complexes, $n\geq 4$. For any $l\geq 2$, 
  \[\E_\sharp^{n+l}(X_1\vee \cdots\vee X_m)\cong\bigoplus_{k=1}^m\E_\sharp^{n+l}(X_k)\oplus\bigoplus_{i\neq j=1}^{m,m}\mathcal{Z}_\sharp^{n+l}(X_i,X_j).\]
  The group $\E_\sharp^{n+l}(X_1)\oplus\cdots\oplus\E_\sharp^{n+l}(X_m)$ is isomorphic to the subgroup given by the diagonal matrix 
  \[\mathrm{diag}\{\E_\sharp^{n+l}(X_1),\cdots,\E_\sharp^{n+l}(X_m)\}.\]
  For $1\leq i\neq j\leq m$, $\mathcal{Z}_\sharp^{n+l}(X_i,X_j)$ is isomorphic to the subgroup \[I_m+E_{ji},\] 
  where $I_m=\mathrm{diag}\{1_{X_1},\cdots,1_{X_m}\}$ is the diagonal matrix of identity maps, $E_{ji}$ is the $m\times m$ matrix with $\mathcal{Z}_\sharp^{n+l}(X_i,X_j)$ in the $(j,i)$-entry and zero otherwise. The addition in $\mathcal{Z}_\sharp^{n+l}(X_i,X_j)$ corresponds to the matrix multiplication in $I_m+E_{ji}$.
\end{theorem}

As an example, we compute the subgroup $\E_\sharp^{n+2}\big(\bigvee_{i=1}^{m} C^{n+2,t_i}_{r_i}\big)$.
\begin{corollary}\label{cor:Ecc:n+2}
 For any $m\geq 1$, there is an isomorphism 
 \[\E_\sharp^{n+2}\big(\bigvee_{i=1}^{m} C^{n+2,t_i}_{r_i}\big)\cong (\zz{})^{m^2}\oplus \bigoplus_{i,j=1}^{m,m}\zz{\min(r_i,t_j)},\]
 where $(\zz{})^{m^2}$ is the direct sum of $m^2$ copies of $\zz{}$.   
\end{corollary}

Let $X$ be an $\an$-complexes with $H_{n+2}(X)=0$, then the naturality of Hurewicz homorphisms implies that $\E_\sharp^{n+2}(X)$ is a normal subgroup of $\E_\ast(X)$ (Lemma \ref{lem:Ecc-H}).
In particular, for $X=\bigvee_{i=1}^{m} C^{n+2,t_i}_{r_i}$, we have 
\begin{theorem}\label{thm:Ecc-H}
 For any $m\geq 1$, there is a short exact sequence
  \[\begin{tikzcd}[column sep=scriptsize]
    \E_\sharp^{n+2}\big(\bigvee_{i=1}^{m} C^{n+2,t_i}_{r_i}\big)\ar[r,tail]& \E_\ast\big(\bigvee_{i}^{m} C^{n+2,t_i}_{r_i}\big)\ar[r,two heads]& (\zz{})^{m^2}\oplus (\zz{})^{m^2}
  \end{tikzcd}.\]
\end{theorem}

The paper is organized as follows. In Section \ref{sec:prelim} we introduce the global conventions and notation in this paper and list homotopy cofibre sequences for indecomposable $\an$-complexes. In Section \ref{sec:tables} we summarize the explicit generators of the groups $[X,Y]$ by two tables (Tables \ref{table1}, \ref{table2}) and a theorem of relation formulas (Theorem \ref{thm:eqs}), where $X,Y$ indecomposable $\an$-complexes, $n\geq 3$. Section \ref{sec:proofs-tables} is devoted to the proofs of results in Section \ref{sec:tables} and Section \ref{sec:htpEq} covers the proofs of the above theorems and corollaries.

\subsection*{Acknowledgement}
The author would like to thank the editor and referee(s) sincerely for their valuable feedback on the earlier versions of the paper. The author was supported by the National Natural Science Foundation of China (Grant no. 12101290) and China Postdoctoral Science Foundation (Grant no. 2021M691441).

\section{Preliminaries}\label{sec:prelim}
This section covers global conventions and notation adopted in this paper and reviews the useful homotopy cofibre sequences for indecomposable $\an$-complexes, $n\geq 3$. 

\subsection{Conventions and notation}\label{sec:notations}
Throughout the paper we assume that all spaces are pointed finite CW-complexes and that all maps between spaces are base-point-preserving ones; we don't distinguish a map from its homotopy class in notation.
Given maps $f\colon X\to Y,g\colon Y\to Z$, denote by $g\circ f$ or $gf$ the composite map $X\to Z$.
Given a map $f\colon X_1\vee\cdots\vee X_k\to Y_1\vee \cdots\vee Y_l$, where $X_1,\cdots,X_k$ and $Y_1,\cdots,Y_l$ be $\an$-complexes, $n\geq 4$, we usually represent $f$ in the matrix form 
\[f=(f_{ij})_{l\times k}=\begin{pmatrix}
 f_{11}&\cdots&f_{1k}\\
 \vdots&\ddots&\vdots\\
 f_{l1}&\cdots&f_{lk}
\end{pmatrix},\]
where $f_{ij}=p_{Y_i}\circ f\circ i_{X_j}$, $i_{X_j}\colon X_j\to X_1\vee \cdots \vee X_k$ and $p_{Y_i}\colon Y_1\vee \cdots \vee Y_l\to Y_i$ are respectively the canonical inclusion and projection maps, $1\leq i\leq l, 1\leq j\leq k$. 
Given an abelian group $A$, denote by $H_n(X;A)$ (resp. $H^n(X;A)$) the $n$-th reduced homology (resp. cohomology) group of $X$ with coefficients in $A$; if $A=\Z$, write $H_n(X)=H_n(X;\Z)$ and $H^n(X)=H^n(X;\Z)$. 
Denote by $A\langle a_1,\cdots, a_n \rangle$ to indicate that the group $A$ is generated by elements $a_1,\cdots, a_n$. For convenient, we use the arrows $``\rightarrowtail"$ and $``\twoheadrightarrow"$ to denote monomorphisms and epimorphisms of groups,  respectively. 
Given $x\in\zz{k},b\in\zz{l}$, $k,l\geq 1$, the sum ``$a+b$" means the sum $p(a)+q(b)$ in some group $\zz{c}$, where $c\leq \min(k,l)$, and $p\colon \zz{k}\to\zz{c}$ and $q\colon \zz{l}\to\zz{c}$ are the canonical projections. 

Unless otherwise stated, we are working in the \emph{stable range} $n\geq 4$. 
Let spaces $X,Y,\cdots$ be indecomposable $\an$-complexes below.
For simplicity, we shall frequently use the notations when they cause no confusion. Denote 
\[M=\mm{r}{n+j}(j=0,1),\quad \uc=C^{n+2}_r,\quad \oc=C^{n+2,t},\quad C=C^{n+2,t}_r\] for different upper and lower indices $r,t\geq 1$. Denote by $1_X$ the identity map of $X$ and denote $k\wedge 1_X\colon S^1\wedge X\to S^1\wedge X$ by $k\cdot 1_X$; in particular, we denote 
\begin{align*}
  &1_n=1_{S^n},\quad 1_M=1_{\mm{r}{n+j}}~(j=0,1), \quad 1_{\eta}=1_{\ceta},\\
&1_{\uc}=1_{C^{n+2}_r},\quad 1_{\oc}=1_{C^{n+2,t}},\quad 1_{C}=1_{\CC}.
\end{align*}
To emphasize the domains for inclusions and the codomains for pinch maps,  we use the same notations $i_A$ and $q_B$ respectively to denote the canonical inclusion $A\to X$ and the canonical pinch map $Y\to B$ for different indecomposable $\an$-complexes $A,B$. Concretely, we agree once for all that 
\begin{enumerate}[i)]
  \item  $i_{n+k}$ and $q_{n+k}$ denote the canonical inclusions: $S^{n+k}\to X$ and the canonical pinch maps: $Y\to S^{n+k}$, respectively. For example, we denote by $i_n\colon S^n\to X$  the canonical inclusions for different spaces  $X=\m{r}{n},\ceta, C^{n+2}_r,C^{n+2,t},\CC$.
  \item $i_M$ and $q_M$ denote the canonical inclusions $\m{r}{n}\to X$ and pinch maps $Y\to \m{t}{n+1}$, respectively. Possible candidates are $X=C^{n+2}_r,\CC;Y=C^{n+2,t},\CC$.
  \item  $i_\eta\colon \ceta\to C^{n+2}_r,i_{\oc}\colon C^{n+2,t}\to \CC$ denote the canonical inclusions, and $q_\eta\colon C^{n+2,t}\to \ceta, q_{\uc}\colon \CC\to C^{n+2}_r$ denote the canonical pinch maps that collapse the subspaces $S^{n+1}$.
  \item if there are simultaneously indecomposable $\an$-complexes $X$ and $X'$ of the same type but with different power indices,  we use the prime symbol $'$ in the above notation of spaces to indicate the different maps; for example, we denote by $i_{M}\colon \m{r}{n}\to C^{n+2}_r$ and $i_{M'}\colon \m{r'}{n}\to C^{n+2}_{r'}$ for the canonical inclusions when $r\neq r'$.  
\end{enumerate}

\subsection{Homotopy cofibre sequences for indecomposable $\an$-complexes}\label{sec:Cofs} The following homotopy cofibre sequences, due to Zhu and Pan \cite{ZP17},  will be frequently used in Section \ref{sec:proofs-tables}. 
  \begin{enumerate}[(a)]
    \item\label{Cof-Moore} Homotopy cofibre sequence for $\mm{r}{n}$: $S^n\xra{p^r}S^n\xra{i_n}\mm{r}{n}\xra{q_{n+1}}S^{n+1}$.
  
    \item\label{Cof-eta}  Homotopy cofibre sequence for $C^{n+2}_{\eta}$: $S^{n+1}\xra{\eta}S^{n}\xra{i_n} C_{\eta}^{n+2}\xra{q_{n+2}} S^{n+2}$.

    \item\label{Cof-uc}  Homotopy cofibre sequences for $C^{n+2}_{r}$:
       \begin{description}
         \item [\textbf{Cof1}] $S^{n}\vee S^{n+1}\xra{(2^r, \eta)}S^{n}\xra{i_n}C^{n+2}_{r}\xra{q_S}S^{n+1}\vee S^{n+2}$;
         \item [\textbf{Cof2}] $S^{n+1}\xra{i_n\eta}\m{r}{n}\xra{i_{M}}C^{n+2}_{r}\xra{q_{n+2}}S^{n+2}$;
         \item [\textbf{Cof3}] $S^{n}\xra{i_n2^r}C_{\eta}^{n+2}\xra{i_{\eta}}C^{n+2}_{r}\xra{q_{n+1}}S^{n+1}$;
       \end{description}

    \item\label{Cof-oc}  Homotopy cofibre sequences for $C^{n+2,t}$:
        \begin{description}
         \item [\textbf{Cof1}] $S^{n+1}\xra{\binom{\eta}{2^t}}S^{n}\vee S^{n+1}\xra{(i_n,i_{n+1})}C^{n+2, t}\xra{q_{n+2}} S^{n+2}$;
         \item [\textbf{Cof2}] $\m{t}{n}\xra{\eta q_{n+1}}S^{n}\xra{i_n}C^{n+2,t}\xra{q_{M}}\m{t}{n+1}$;
         \item [\textbf{Cof3}] $C_{\eta}^{n+1}\xra{2^tq_{n+1}}S^{n+1}\xra{i_{n+1}}C^{n+2, t}\xra{q_{\eta}}C_{\eta}^{n+2}$;
       \end{description}

    \item\label{Cof-cc}  Homotopy cofibre sequences for $ C_{r}^{n+2,t}$:
         \begin{description}
         \item [\textbf{Cof1}] $S\xra{\ma{r}{t}}S\xra{i_S}C^{n+2,t}_{r}\xra{q_{\Sigma S}}\Sigma S$, where $S=S^{n}\vee S^{n+1}$;
         \item [\textbf{Cof2}] $\m{t}{n}\xra{i_n\eta q_{n+1}}\m{r}{n}\xra{i_{M}}C^{n+2,t}_{r}\xra{q_{M}}\m{t}{n+1}$;
         \item [\textbf{Cof3}] $S^{n}\vee \m{t}{n}\xra{(2^r, \eta q_{n+1})}S^{n}\xra{i_n}C^{n+2,t}_{r}\xra{\binom{q_{n+1}}{q_M}}S^{n+1}\vee\m{t}{n+1}$;
         \item [\textbf{Cof4}] $S^{n+1}\xra{\binom{i_n\eta}{2^t}}\m{r}{n}\vee S^{n+1}\xra{(i_M,i_n)}C^{n+2,t}_{r}\xra{q_{n+2}}S^{n+2}$;
         \item [\textbf{Cof5}] $C_{r}^{n+1}\xra{2^tq_{n+1}}S^{n+1}\xra{i_{n+1}}C^{n+2,t}_{r}\xra{q_{\uc}}C_{r}^{n+2}$;
         \item [\textbf{Cof6}] $S^{n}\xra{i_n2^r}C^{n+2,t}\xra{i_{\oc}}C^{n+2,t}_{r}\xra{q_{n+1}}S^{n+1}$.

       \end{description}
  \end{enumerate}

The following lemma is trivial from the homotopy cofibre sequences above.
\begin{lemma}\label{lem:cofib}
Let $n\geq 3$ and $k\in \{0,1\}$.
 \begin{enumerate}
   \item For the canonical inclusions $i_{n+k}$ from $S^{n+k}$ into $C^{n+2,t}$ or $\CC$, there holds 
\[i_n\eta =2^t\cdot i_{n+1}.\]
   \item For the canonical pinch maps $q_{n+k}$ from  $C^{n+2}_r$ or $\CC$ onto $S^{n+k}$, there holds
   \[\eta q_{n+2} =2^r\cdot q_{n+1}.\]
 \end{enumerate}
\end{lemma}

\section{Groups $[X,Y]$ and their explicit generators}\label{sec:tables}
Recall that  there is a contravariant functor $D=D_{2n+2}$, called the \emph{\SW duality functor} (cf. \cite{Bauesbook}), defined on the stable homotopy category of $\an$-complexes. In particular, the functor $D$ induces an isomorphism $[X,Y]\cong [DY,DX]$.
For indecomposable $\an$-complexes, we have
\begin{align*}
&DS^{n+i}=S^{n+2-i}, ~i=0,1,2;\quad D\mm{r}{n}=\mm{r}{n+1};\\
& D\ceta=\ceta,~,DC^{n+2,t}=C^{n+2}_t,~D\CC=C^{n+2,r}_t,~r,t\geq 1.
\end{align*}
 It is also well-known that the canonical inclusion $i_{n+k}\colon S^{n+k}\to X$ is \SW $(2n+2)$-dual to the canonical pinch maps $q_{n+2-k}\colon DX\to S^{n+2-k}$ for indecomposable $\an$-complexes $X$, $n\geq 4$ and $k=0,1,2$. The \SW duality will effectively reduce the computation process to the homotopy classification of based maps between $\an$-complexes.

By \cite{BH91},  there are maps $\tilde{\eta}\in [S^{n+2},\m{}{n}]$ and $\bar{\eta}\in [\m{}{n+1},S^n]$ satisfying the formulas:
\begin{equation}\label{eq:eta}
  \begin{aligned}
 q_{n+1}\tilde{\eta}=\eta=\bar{\eta}i_{n+1};\quad 
 2\bar{\eta}=\eta^2 q_{n+2},\quad 2\tilde{\eta}=i_n\eta^2.
\end{aligned}
\end{equation}
By \citep[Theorem 1.4.4]{Bauesbook}, for each prime $p$ and positive integers $r,t$, there is a map 
\begin{equation*}
  B(\chi^r_t)=B_n(\chi^r_t)\colon \mm{r}{n}\to \mm{t}{n}
\end{equation*}
characterized by the following two properties: 
\begin{enumerate}[(i)]
  \item $H_n(B(\chi^r_t))=\chi^r_t\colon \Z/{p^r}\to\Z/p^{t}$ satisfying $\chi^r_t(1)=1$ if $r\geq t$, otherwise $\chi^r_t(1)=p^{t-r}$;
  \item $B_n(\chi^r_t)=\Sigma B_{n-1}(\chi^r_t)$ under the suspension, $n\geq 2$.
\end{enumerate} 
Consequently, $B(\chi^r_t)$ satisfies the formulas:
\begin{equation}\label{eq:B}
  \begin{aligned}
    r\geq t:~~B(\chi^r_t)i_n=i_n,\quad q_{n+1}B(\chi^r_t)=p^{r-t} q_{n+1};\\
    r\leq t:~~B(\chi^r_t)i_n=p^{t-r}i_n,\quad q_{n+1}B(\chi^r_t)= q_{n+1}. 
  \end{aligned} 
\end{equation}

 In addition to $B(\chi^r_t)$, there are other newly defined generators, which can be characterized by relation formulas in the following theorem.
\begin{theorem}\label{thm:eqs}
Let $n\geq 3$ and $r,t\geq 1$ be integers.
  \begin{enumerate}[1.] 
   \item\label{eq:zeta} $\tilde{\zeta}\in [S^{n+2},\ceta]$ and  $\bar{\zeta}\in [\ceta,S^n]$ satisfy the formulas:
    \begin{equation*}
      \begin{split}
        q_{n+2}\tilde{\zeta}=2\cdot 1_{n+2},&\quad \bar{\zeta}i_n=2\cdot 1_n,\\
        \tilde{\zeta}q_{n+2}+i_n\bar{\zeta}&=2\cdot 1_\eta.
      \end{split}
    \end{equation*}
 
    \item\label{eq:xi} $\tilde{\xi}_t\in [\m{t+1}{n},C^{n+2,t}]$ and $\bar{\xi}_r\in [C^{n+2}_r,\m{r+1}{n}]$ satisfy the formulas:
    \begin{align*}
      \tilde{\xi}_ti_{n+1}&=i_{n+1},\quad q_\eta \tilde{\xi}_t=\tilde{\zeta}q_{n+2},\quad q_{M}\tilde{\xi}_t=B(\chi^{t+1}_t),\quad q_{n+2}\tilde{\xi}_t=2q_{n+2};\\
    q_{n+1}\bar{\xi}_r&=q_{n+1},\quad  \bar{\xi}_ri_\eta=i_n\bar{\zeta},\quad \bar{\xi}_ri_M=B(\chi^{r}_{r+1}),\quad \bar{\xi}_ri_n=2i_n.
    \end{align*}

    \item\label{eq:vartheta} If $t\leq t'$, $\tilde{\vartheta}^{t}_{t'}\in [C^{n+2,t},C^{n+2,t'}]$ satisfies $\tilde{\vartheta}^t_t=1_{\oc}$ and the formulas:
    \begin{align*}
      \tilde{\vartheta}^{t}_{t'}i_n=i_n,\quad q_{M'}\tilde{\vartheta^{t}_{t'}}=B(\chi^{t}_{t'})q_M, \quad \tilde{\vartheta}^{t}_{t'}i_{n+1}=2^{t'-t}i_{n+1}, \quad q_\eta \tilde{\vartheta}^{t}_{t'}=q_\eta.
    \end{align*}
    If $r\geq r'$, $\bar{\vartheta}^r_{r'}\in [C^{n+2}_r,C^{n+2}_{r'}]$ satisfies $\bar{\vartheta}^r_{r}=1_{\uc}$ and the formulas:
    \begin{align*}
      q_{n+2}\bar{\vartheta}^r_{r'}=q_{n+2},\quad  \bar{\vartheta}^r_{r'}i_M=i_{M'}B(\chi^r_{r'}),\quad q_{n+1}\bar{\vartheta}^r_{r'}=2^{r-r'}q_{n+1}, \quad \bar{\vartheta}^r_{r'}i_\eta=i_\eta.
    \end{align*} 

    \item\label{eq:CC}
 If $t'\geq t, r\geq r'$, $L(\chi)\in [\CC, C^{n+2,t'}_{r'}]$ satisfies $L(\chi)=1_{C}$ for $t'=t,r'=r$ and the following formulas:
    \begin{align*}
      L(\chi) i_{\oc}=i_{\oc'}\tilde{\vartheta}^{t}_{t'},&\quad q_{\uc'}L(\chi)=\bar{\vartheta}^r_{r'}q_{\uc};\\
       L(\chi)i_M=i_{M'}B(\chi^r_{r'}),&\quad q_{M'}L(\chi)=B(\chi^{t}_{t'})q_M.
    \end{align*}
    \end{enumerate}
\end{theorem}

In terms of the above conventions and notation,
we summarize the group structure of $[X,Y]$ of indecomposable $\an$-complexes $X,Y$ with $n\geq 3$ and their explicit generators in Tables \ref{table1}, \ref{table2}. The groups $[X,Y]$ in which $X$ or $Y$ is a mod $p^r$ Moore space with $p$ an odd prime can be easily computed. We omit the discussion of these groups in this paper.

  \begin{amssidewaystable}
    \centering
    \caption{Maps between indecomposable $\an$-complexes, I}\label{table1}
   
\begin{tabular}{c| c| c| c|c|c}
\hline
  &$S^{n}$& $S^{n+1}$& $S^{n+2}$ & $M_{2^r}^n $ & $M_{2^t}^{n+1}$
\\
\hline
$S^{n}$&$\Z~ 1_n$&$\Z/2~\eta$&$\Z/2 ~\eta^2$&$\Z/2~\eta q_{n+1}$
&\tabincell{l}{$t=1: \Z/4~ \bar{\eta}$\\$t>1:\Z/2\oplus \zz{}~\bar{\eta}B(\chi^t_1),~\eta^2q_{n+2}$}
\\
\hline
$S^{n+1}$&$0$&$\Z~ 1_{n+1}$&$\Z/2~\eta $&$\Z/{2^r}~ q_{n+1}$&$\Z/2 ~\eta q_{n+2}$ \\
\hline
$S^{n+2}$&$0$&$0$&$\Z~1_{n+2}$&$0$&$\zz{t} ~q_{n+2}$\\
\hline
$M_{2^{r'}}^n $ &$\Z/{2^{r'}} ~i_n$&$\Z/2~ i_n\eta$
 & \tabincell{l}{$r'=1:\Z/4 ~\tilde{\eta}$\\[1ex]$r'>1:\Z/2\oplus \Z/2$\\\hspace*{1cm}$B(\chi^1_{r'})\tilde{\eta},~ i_n\eta^2$} &
 \tabincell{l}{$r=r'=1:\Z/4~1_M$\\[1ex] otherwise:\\ \hspace*{1cm}$\Z/2^l \oplus \Z/2$\\ \hspace*{1cm}$B(\chi^{r}_{r'}),~ i_n\eta q_{n+1}$} &
 \tabincell{l}{$t=1=r':\Z/2\oplus\Z/2$\quad $ i_n\bar{\eta}, ~\tilde{\eta}q_{n+2}$\\[1ex]
 $t>1=r':\Z/2 \oplus\Z/4$\quad $i_n\bar{\eta}B(\chi^t_1),~\tilde{\eta}q_{n+2}$\\[1ex]
 $t=1<r': \Z/4\oplus\Z/2$\quad $i_n\bar{\eta}, ~B(\chi^t_1)\tilde{\eta}q_{n+2}$\\[1ex]
 $t>1<r':\Z/2\oplus\Z/2\oplus\Z/2$\\\hspace*{1cm}$ i_n\bar{\eta}B(\chi^t_1), ~B(\chi^t_1)\tilde{\eta}q_{n+2},~i_n\eta^2 q_{n+2}$
 }\\
\hline
$M_{2^{t'}}^{n+1}$&$0$&$\Z/2^{t'}~i_{n+1}$&$\Z/2~i_{n+1}\eta$&$\Z/2^m~i_{n+1}q_{n+1}$&
 \tabincell{l}{$t=t'=1$: $\Z/4~1_M$\\otherwise:
 $\Z/2^n\oplus \Z/2$\\\hspace*{1.5cm}$B(\chi^t_{t'}),~i_{n+1}\eta q_{n+2}$}\\
 \hline\hline
 $C^{n+2}_{\eta}$ & $\Z~i_n$ & $0$ & $\Z~\tilde{\zeta}$ &$0$  &  $\zz{t}~\tilde{\zeta}q_{n+2} $ \\
\hline
$C^{n+2, t'}$& $\Z~i_n$ & $\Z/{2^{t'+1}}~i_{n+1}$ &$\Z/{2}~i_{n+1}\eta$ &$\Z/{2^{m^{''}}}~i_{n+1}q_{n+1}$ & \tabincell{c}{$\Z/{2^{n^{''}}}\oplus \Z/2$\\\hspace*{1cm}$\tilde{\xi}_{t'} B(\chi^{t}_{t'+1}),~i_{n+1}\eta q_{n+2}$}\\
\hline
$C^{n+2}_{r'}$&$\Z/{2^{r'}}~i_n$&$0$ &\tabincell{c}{$\Z\oplus \Z/{2}$\\\hspace*{1cm}$i_{\eta}\tilde{\zeta},~i_{M'}B(\chi^1_{r'})\tilde{\eta}$}  &$\Z/{2^l}~i_{M'}B(\chi^{r}_{r'})$&\tabincell{c}{ $\zz{t}\oplus \Z/2$\\\hspace*{1cm}$i_{\eta}\tilde{\zeta}q_{n+2},~i_{M'}B(\chi^1_{r'})\tilde{\eta}q_{n+2}$}\\
\hline
$C^{n+2, t'}_{r'}$&$\Z/{2^{r'}}~i_n$ &$\Z/{2^{t'+1}}~i_{n+1}$ & \tabincell{c}{$\Z/{2}\oplus  \Z/{2}$\\\hspace*{1cm}$i_{n+1}\eta,~i_{M'}B(\chi^1_{r'})\tilde{\eta}$}&\tabincell{c}{$\Z/{2^l}\oplus \Z/{2^{m^{''}}}$\\\hspace*{1cm}$i_{M'}B(\chi^r_{r'}),~i_{n+1}q_{n+1}$ }&
\tabincell{c}{ $\Z/{2^{n^{''}}}\oplus \Z/2\oplus \Z/2$\\ \hspace*{1cm}$i_{\oc'}\tilde{\xi}_{t'} B(\chi^t_{t'+1}), ~i_{M'}B(\chi^1_{r'})\tilde{\eta}q_{n+2}$,\\\hspace*{2.5cm}~$i_{n+1}\eta q_{n+2}  $ }\\
\hline
\end{tabular}
\\[1ex]
$j=\max(t,r'),~k=\min(t,r');\hspace{3mm}
l=\min(r,r'),~ l'=\min(r+1,r'); ~~$\\
$m=\min(r,t'), ~m'=\min(r+1,t'),m^{''}=\min(r,t'+1);\hspace{3mm}n=\min(t,t'),~ n^{''}=\min(t,t'+1)
$

\end{amssidewaystable}

   \begin{amssidewaystable}

 \centering
 \caption{Maps between indecomposable $\an$-complexes, II}\label{table2}

\begin{tabular}{c| c| c| c|c}
 \hline 
  &$\ceta$& $C^{n+2,t}$& $C^{n+2}_r$& $C^{n+2,t}_r$ \\
\hline
$S^{n}$&$\Z~\bar{\zeta}$&$\Z~\bar{\zeta}q_\eta~\oplus\Z/2~\bar{\eta}B(\chi^t_1)q_M$&$\Z/2~ \eta q_{n+1}$&$\Z/2\oplus \Z/2~\bar{\eta}B(\chi^t_1)q_M,~\eta q_{n+1}$\\
\hline
$S^{n+1}$&$0$&$0$&$\zz{r+1}~q_{n+1} $&$\zz{r+1}~q_{n+1} $ \\
\hline
$S^{n+2}$&$\Z~ q_{n+2}$&$\zz{t}~ q_{n+2}$&$\Z~q_{n+2}$&$\zz{t}~q_{n+2}$\\
\hline
$M_{2^{r'}}^n $ &\tabincell{c}{$\zz{r'}$\\$i_n\bar{\zeta}$}&\tabincell{c}{$\Z/{2^{r'}}\oplus \Z/2$\\ $i_n\bar{\zeta}q_\eta,~i_n\bar{\eta}B(\chi^t_1)q_M$}&\tabincell{c}{$\zz{l'}\oplus\Z/2$\\ $B(\chi^{r+1}_{r'})\bar{\xi}_r, ~ i_n\eta q_{n+1}$}
 & \tabincell{l}{$\zz{l'}\oplus \Z/2\oplus\Z/2$\\$B(\chi^{r+1}_{r'})\bar{\xi}_r q_{\uc},~i_n\bar{\eta}B(\chi^t_1)q_M,~i_n\eta q_{n+1}$} \\
 \hline
 $M_{2^{t'}}^{n+1}$&$0$&\tabincell{c}{$\Z/2^{n}$\\$B(\chi^t_{t'})q_M$}&\tabincell{c}{$\zz{m'}$\\$i_{n+1}q_{n+1}$}&\tabincell{c}{$\Z/2^{m'}\oplus\zz{n}$\\$i_{n+1}q_{n+1},~B(\chi^t_{t'})q_M$}
 \\ \hline\hline
 $C^{n+2}_{\eta}$ & \tabincell{l}{$\Z\oplus\Z$\\ $1_{\eta},~i_n\bar{\zeta}$} & \tabincell{c}{$\Z\oplus\zz{t}$\\$q_\eta ~ \tilde{\zeta}q_{n+2}$} & \tabincell{c}{$\Z$\\$\tilde{\zeta}q_{n+2}$}   & \tabincell{c}{ $\zz{t}$\\$\tilde{\zeta}q_{n+2} $} \\
\hline
$C^{n+2, t'}$&\tabincell{c}{$\Z$\\$i_n\bar{\zeta}$}& \tabincell{l}{$t>t':~\Z\oplus\zz{t'+1}$\\\hspace*{0.7cm}$i_n\bar{\zeta}q_\eta, ~\tilde{\xi}_{t'} B(\chi^t_{t'+1})q_M$ \\[1ex]$t\leq t':~\Z\oplus\zz{t}$\\\hspace*{1cm} $~\tilde{\vartheta}^t_{t'}, ~\tilde{\xi}_{t'} B(\chi^t_{t'+1})q_M$}&\tabincell{c}{$\zz{m+1}$\\$i_{n+1}q_{n+1}$ }
&\tabincell{l}{$\Z/{2^{n^{''}}}\oplus\zz{m+1}$\\$\tilde{\xi}_{t'} B(\chi^t_{t'+1})q_Mq_M,~i_{n+1}q_{n+1}$}\\
\hline
$C^{n+2}_{r'}$&\tabincell{l}{$\Z\oplus\Z/{2^{r'}}$\\$i_\eta,~i_n\bar{\zeta}$}&
\tabincell{l}{$t\geq r':~\zz{t+1}\oplus\zz{r'}$\\\hspace*{1cm}$i_\eta q_\eta, ~i_n\bar{\zeta}q_\eta $\\[1ex]$t<r':~\zz{r'+1}\oplus\zz{t}$\\\hspace*{1cm}$i_\eta q_\eta ,~i_\eta\tilde{\zeta}q_{n+2} $}
 &\tabincell{l}{$r'>r:\Z\oplus\zz{r+1}$\\\hspace*{1cm}$i_\eta\tilde{\zeta}q_{n+2},~i_{M'}B(\chi^{r+1}_{r'})\bar{\xi}_r $ \\[1ex]$r'\leq r:\Z\oplus\zz{r'}$\\\hspace*{1cm}$\bar{\vartheta}^r_{r'},~i_{M'}B(\chi^{r+1}_{r'})\bar{\xi}_r  $} &\tabincell{l}{$r\geq r'\leq t:~\zz{t+1}\oplus\zz{r'}$\\\hspace*{1cm}$\bar{\vartheta}^r_{r'} q_{\uc},~i_{M'}B(\chi^{r+1}_{r'})\bar{\xi}_r  q_{\uc}$\\[1ex]$r\geq r'>t:~\zz{r'+1}\oplus\zz{t}$\\\hspace*{1cm}$\bar{\vartheta}^r_{r'} q_{\uc},~i_\eta\tilde{\zeta}q_{n+2}$\\[1ex]$r<r':~\zz{r+1}\oplus\zz{t}$\\\hspace*{1cm}$i_{M'}B(\chi^{r+1}_{r'})\bar{\xi}_r q_{\uc},~i_\eta\tilde{\zeta}q_{n+2}$}\\
\hline
$C^{n+2, t'}_{r'}$&\tabincell{c}{$\Z/{2^{r'}}$\\$i_n\bar{\zeta}$} &
\tabincell{l}{$ t'\geq t<r':\zz{r'+1}\oplus \zz{t}$\\\hspace*{1cm}$i_{\oc'}\tilde{\vartheta}^t_{t'} ,~i_{\oc'}\tilde{\xi}_{t'} B(\chi^t_{t'+1})q_M$\\[1ex]$t'\geq t\geq r':\zz{t+1}\oplus\zz{r'}$\\\hspace*{1cm}$i_{\oc'}\tilde{\vartheta}^t_{t'},~i_n\bar{\zeta}q_\eta $\\[1ex]$t'<t:\zz{t'+1}\oplus\zz{r'}$\\\hspace*{1cm}$i_{\oc'}\tilde{\xi}_{t'} B(\chi^t_{t'+1})q_M,~i_n\bar{\zeta}q_\eta $} & \tabincell{c}{$\Z/{2^{m+1}}\oplus  \Z/{2^{l'}}$\\$i_{n+1}q_{n+1},~i_{M'}B(\chi^{r+1}_{r'})\bar{\xi}_r $}
& \tabincell{l}{$r'>r\vee t'<t:~\zz{m+1}\oplus\zz{l'}\oplus \zz{n^{''}}$\\\hspace*{1cm}$ i_{n+1}q_{n+1},~ i_{M'}B(\chi^{r+1}_{r'})\bar{\xi}_r q_{\uc}$, \\\hspace*{3cm}~$i_{\oc'}\tilde{\xi}_{t'} B(\chi^t_{t'+1})q_M$ \\[1ex]$t'\geq t<r'\leq r:~\zz{m+1}\oplus\zz{r'+1}\oplus\zz{t}$\\\hspace*{1cm}$i_{n+1}q_{n+1},~L(\chi),~i_{\oc'}\tilde{\xi}_{t'} B(\chi^t_{t'+1})q_M$ \\[1ex]
$t'\geq t\geq r'\leq r:~\zz{m+1}\oplus\zz{t+1}\oplus\zz{r'}$\\\hspace*{1cm}$~i_{n+1}q_{n+1},~L(\chi),~i_{M'}B(\chi^{r+1}_{r'})\bar{\xi}_r q_{\uc} $}\\[1ex]
\hline
\end{tabular}
\\[1ex]
$j=\max(t,r'),~k=\min(t,r');\hspace{3mm}
l=\min(r,r'),~ l'=\min(r+1,r'); $\\
$m=\min(r,t'), ~m'=\min(r+1,t'),m^{''}=\min(r,t'+1);\hspace{3mm}n=\min(t,t'),~ n^{''}=\min(t,t'+1)$.

\end{amssidewaystable}

\clearpage

\section{Proofs of groups and generators in Tables \ref{table1},\ref{table2}}\label{sec:proofs-tables}

This section is devoted to the proofs of results in Section \ref{sec:tables}. 
 The entries above the double lines in Table \ref{table1} were proved by Brown and Copeland \cite{BC59} and the notation of the generators are due to Baues and Hennes \cite{BH91}. We shall directly use relations (\ref{eq:eta}) and (\ref{eq:B}) in our computations. In the remainder of this section, let $n\geq 4$ and let $X,Y$ be any an indecomposable $\an$-complex, we shall compute the groups $[X,Y]$ one-by-one and determine their explicit generators. Recall that in stable range a homotopy cofibre sequence 
\[X\xra{f}Y\xra{i}Z\xra{q}\Sigma X\xra{\Sigma f}\Sigma Y\]
is also a homotopy fibre sequence.

\begin{proposition}\label{prop:X=spheres}
 The groups $\pi_{n+i}(X)$ with $i\leq 2$ are given by 
 \begin{table}[H]
  \begin{tabular}{c|c|c|c}
    \hline
    &$S^{n}$&$S^{n+1}$&$S^{n+2}$\\
    \hline
    $\ceta$&$\Z~i_n$&$0$&$\Z~\tilde{\zeta}$ \\
    \hline
    $C^{n+2,t}$&$\Z~i_n$&$\zz{t+1}~i_{n+1}$&$\zz{}~i_{n+1}\eta$ \\
    \hline
    $C^{n+2}_r$&$\zz{r}~i_n$&$0$&$\Z\oplus\zz{}~~i_\eta\tilde{\zeta},i_MB(\chi)\tilde{\eta}$
    \\ \hline
    $\CC$&$\zz{r}~i_n$&$\zz{t+1}~i_{n+1}$&$\zz{}\oplus\zz{}~~~i_{n+1}\eta,i_MB(\chi)\tilde{\eta}$\\
    \hline
    \end{tabular}
    \caption{$\pi_{n+i}(X)$ for $i\leq 2$}\label{table:SC}
 \end{table}
 \noindent Here $\tilde{\zeta}\colon S^{n+2}\to \ceta$ is a generator satisfying the formula
 \begin{equation}\label{eq:zeta'-def}
 q_{n+2}\tilde{\zeta}= 2\cdot 1_{n+2}.
 \end{equation}
    \begin{proof}
    The group $\pi_{n+2}(\ceta)$ and the generator $\tilde{\zeta}$ are due to Toda \citep[Section 8.1]{Toda2}. Other homotopy groups $\pi_{n+i}(X)$ can be easily computed and are due to \citep[page 301]{ZP17}. 
    \end{proof}
\end{proposition}

By the \SW duality, we get the cohomotopy groups $[X,S^{n+i}]$ and their generators for $i=0,1,2$.

  \begin{proposition}\label{prop:MC}
   Let $l=\min(r,r'),m''=\min(r,t'+1),n''=\min(t,t')$. 
     The groups $[\m{r}{n},X]$ and $[\m{t}{n+1},X]$ are given by 
    \begin{table}[H]
      \centering
    \begin{tabular}{c|c|c}
    \hline
    &$\m{r}{n}$&$\m{t}{n+1}$\\
    \hline
    $\ceta$&$0$&$\zz{t}~\tilde{\zeta}q_{n+2}$ \\
    \hline
    $C^{n+2,t'}$&$\zz{m''}~i_{n+1}q_{n+1}$&$\zz{n''}\oplus \zz{}~~\tilde{\xi}_{t'}B(\chi^t_{t'+1}),~i_{n+1}\eta q_{n+2}$\\
    \hline
    $C^{n+2}_{r'}$&$\zz{l}~i_{M'}B(\chi^r_{r'})$&$\zz{t}\oplus\zz{}~~i_\eta \tilde{\zeta}q_{n+2},~i_{M'}B(\chi^1_{r'})\tilde{\eta}q_{n+2}$
    \\ \hline
    $	C^{n+2,t'}_{r'}$& \tabincell{c}{ $\zz{l}\oplus\zz{m''}$\\$i_{M'}B(\chi^r_{r'}),~i_{n+1}q_{n+1}$}&\tabincell{c}{$\zz{n''}\oplus\zz{}\oplus\zz{}$\\$i_{\oc'}\tilde{\xi}_{t'}B(\chi^t_{t'+1}),~i_{M'}B(\chi^1_{r'})\tilde{\eta}q_{n+2},~i_{n+1}\eta q_{n+1}$}    \\
    \hline
    \end{tabular}
 \caption{$[\m{r}{n+i},X]$, $i=0,1$}\label{table:MC}
    \end{table}
   \noindent Here $\tilde{\xi}_{t'}\colon\m{t'+1}{n+1}\to C^{n+2,t'}$ satisfies the formula $\tilde{\xi}_{t'}i_{n+1}=i_{n+1}.$
    \begin{proof}
    (1) $\ul{\ceta}$. Applying  $[-,\ceta]$ to the homotopy cofibre sequences for $\m{r}{n}$, $\m{t}{n+1}$, respectively, there are exact sequence of groups:
    \begin{align*}
    &0= [S^{n+1},\ceta]\xra{q_{n+2}^*}[\m{r}{n},\ceta]\xra{i_n^*}[S^n,\ceta]\xra{2^r}[S^n,\ceta];\\
    &[S^{n+2},\ceta]\xra{2^t}[S^{n+2},\ceta]\xra{q_{n+2}^*}[\m{t}{n+1},\ceta]\to [S^{n+1},\ceta]=0.
    \end{align*}
 Then it follows that $[\m{r}{n},\ceta]=0,[\m{t}{n+1},\ceta]\cong\zz{t}\langle \tilde{\zeta}q_{n+2}\rangle$.

    (2) $\ul{C^{n+2,t'}}$. There is an exact sequence by applying $[\m{r}{n},-]$ to \textbf{Cof3}:
    \[[\m{r}{n},C^{n+1}_\eta]\xra{(2^{t'}q_{n+1})_*}[\m{r}{n},S^{n+1}]\xra{(i_{n+1})_*}[\m{r}{n},C^{n+2,t'}]\to [\m{r}{n},\ceta]=0,\]
    where $[\m{r}{n},C^{n+1}_\eta]\cong \zz{r}\langle \Sigma^{-1}(\tilde{\zeta}q_{n+2})\rangle$. By (\ref{eq:zeta'-def}) we then have
    \[[\m{r}{n},C^{n+2,t'}]\cong\zz{\min(r,t'+1)}\langle i_{n+1}q_{n+1}\rangle.\]
    For the group $[\m{t}{n+1},C^{n+2,t'}]$,
    $[S^{n+1},C^{n+2,t'}]\cong\zz{t'+1}\langle i_{n+1}\rangle$ implies that there exists an extension $\tilde{\xi}_{t'}\colon \m{t'+1}{n+1}\longrightarrow C^{n+2,t'}$ such that $\tilde{\xi}_{t'}i_{n+1}=i_{n+1}$.
    Let
    \[f=\big(i_{n+1}(\Sigma\bar{\eta}),\tilde{\xi}_{t'}\big)\colon \m{}{n+2}\vee \m{t'+1}{n+1}\to C^{n+2,t'},\]
    then one checks that $f_\sharp\colon \pi_{n+i}(\m{}{n+2}\vee \m{t'+1}{n+1})\to \pi_{n+i}(C^{n+2,t'})$ is an isomorphism if $i=1$ and an epimorphism if $i=2$.
    For simplicity we write $X=\m{}{n+2}\vee \m{t'+1}{n+1}, \oc=C^{n+2,t'}$.
    Consider the commutative diagram induced by $f$, in which rows are exact sequences: 
    \[\begin{tikzcd}[sep=scriptsize]
      \pi_{n+2}(X)\ar[r,"2^t"]\ar[d,"f_\sharp"]&\pi_{n+2}(X)\ar[r,"q_{n+2}^\ast"]\ar[d,"f_\sharp"]&{[\m{t}{m+1},X]}\ar[r,"i_{n+1}^\ast"]\ar[d,"f_\sharp"]&\pi_{n+1}(X)\ar[r,"2^t"]\ar[d,"f_\sharp"]&\pi_{n+1}(X)\ar[d,"f_\sharp"]\\
      \pi_{n+2}(\oc)\ar[r,"2^t"]&\pi_{n+2}(\oc)\ar[r,"q_{n+2}^\ast"]&{[\m{t}{m+1},\oc]}\ar[r,"i_{n+1}^\ast"]&\pi_{n+1}(\oc)\ar[r,"2^t"]&\pi_{n+1}(\oc)
    \end{tikzcd}\]
   After computing kernels and cokernels of the four multiplications $2^t$, we get the commutative diagram of exact rows and columns:
    \[\begin{tikzcd}[sep=scriptsize]
     \Z/2\langle i_{n+2}+ i_{n+1}\eta\rangle \ar[d,tail] \ar[r,"q_{n+2}^*",tail] & \ker(f_\sharp) \ar[d,tail] \ar[r] & 0 \ar[d]&\\
     \Z/2\langle i_{n+2}\rangle \oplus \Z/2\langle i_{n+1}\eta\rangle \ar[d,"f_\sharp",two heads] \ar[r,"q_{n+2}^*",tail] & {[\m{t}{n+1},X]} \ar[d,"f_\sharp",two heads] \ar[r,two heads] & K \ar[d,"\cong"] \\
     \Z/2\langle i_{n+1}\eta\rangle  \ar[r,"q_{n+2}^*",tail] & {[\m{t}{n+1},\oc]} \ar[r] & K
    \end{tikzcd}\]
   where  $K=\ker\big(\zz{t'+1}\langle i_{n+1}\rangle \xra{2^t} \zz{t'+1}\langle i_{n+1}\rangle\big)$. Then by the exactness of the middle column and the group 
    \begin{align*}
    [\m{t}{n+1},X]\cong& \Z/2 \langle i_{n+2}q_{n+2}\rangle \oplus \zz{\min({t,t'+1})}\langle B(\chi^{t}_{t'+1}) \rangle \oplus \Z/2 \langle i_{n+1}\eta q_{n+2}\rangle\\
    \cong&\zz{\min({t,t'+1})}\langle B(\chi^t_{t'+1}) \rangle \oplus \Z/2 \langle i_{n+1}\eta q_{n+2}\rangle\oplus\Z/2 \langle i_{n+2}q_{n+2}+i_{n+1}\eta q_{n+2}\rangle
    \end{align*}
    we get 
    $[\m{t}{n+1},\oc]\cong \zz{\min({t,t'+1})}\langle \tilde{\xi}_{t'} B(\chi^t_{t'+1})\rangle \oplus \Z/2 \langle i_{n+1}\eta q_{n+2}\rangle.$

    (3) $\ul{C^{n+2}_{r'}}$. There is an exact sequence by applying $[\m{r}{n},-]$ to \textbf{Cof2}:
    \[[\m{r}{n},S^{n+1}]\xra{i_n\eta}[\m{r}{n},\m{r'}{n}]\xra{i_M}[\m{r}{n},C^{n+2}_{r'}]\to[\m{r}{n},S^{n+2}]=0.\]
    Then
    $[\m{r}{n},C^{n+2}_{r'}]\cong\Z/2^{\min(r,r')}\langle i_{M'}B(\chi^r_{r'})\rangle$ then follows from the group structure and generators of $[\m{r}{n},\m{r'}{n}]$.

   The group $[\m{t}{n+1},C^{n+2}_{r'}]$ follows from the exact sequence:
    \[[S^{n+2},C^{n+2}_{r'}]\xra{2^t}[S^{n+2},C^{n+2}_{r'}]\xra{q_{n+2}^*}[\m{t}{n+1},C_{r'}^{n+2}]\to [S^{n+1},C^{n+2}_{r'}]=0.\]
  
    (4) $\ul{C^{n+2,t'}_{r'}}$. Let
    $g=(i_{M},i_{\oc})\colon \m{r'}{n}\vee C^{n+2,t'}\to C^{n+2,t'}_{r'}$,  then \[g_\ast\colon \pi_{n+i}(\m{r'}{n}\vee C^{n+2,t'})\to \pi_{n+i}(C^{n+2,t'}_{r'})\] is an epimorphism for $i=0,1$.
   For short write $Y=\m{r'}{n}\vee C^{n+2,t'}$ and $C'=C^{n+2,t'}_{r'}$. Consider the following commutative diagram induced by $g$ with exact rows:
   \[\begin{tikzcd}[sep=scriptsize]
    \pi_{n+1}(Y)\ar[r,"2^r"]\ar[d,"g_\ast"]&\pi_{n+1}(Y)\ar[r,"q_{n+1}^\ast"]\ar[d,"g_\ast"]&{[\m{r}{n},Y]}\ar[r,"i_n^\ast"]\ar[d,"g_\ast"]&\pi_n(Y)\ar[r,"2^r"]\ar[d,"g_\ast"]&\pi_n(Y)\ar[d,"g_\ast"]\\
    \pi_{n+1}(C')\ar[r,"2^r"]&\pi_{n+1}(C')\ar[r,"q_{n+1}^\ast"]&{[\m{r}{n},C']}\ar[r,"i_n^\ast"]&\pi_n(C')\ar[r,"2^r"]&\pi_n(C')
   \end{tikzcd}\]
  By similar arguments as that in the proof of $[\m{t}{n+1},C^{n+2,t'}]$, there is a commutative diagram of exact sequences:
  \[\begin{tikzcd}[sep=scriptsize]
    \zz{}\langle (i_n\eta,2^{\min(r-1,t')}\{i_{n+1}\})\rangle\ar[r,"q_{n+1}^\ast",tail]\ar[d]&\ker(g_\ast)\ar[d]\ar[r]&0\ar[d]\\
    \Z/2\langle i_n\eta\rangle \oplus \Z/{2^{\min(r,t'+1)}}\langle \{ i_{n+1}\}\rangle\ar[r,"q_{n+1}^*",tail]\ar[d,"g_\ast",two heads]& {[\m{r}{n},Y]} \ar[d,"g_\ast"]\ar[r,"i_n^*",two heads]& K_1\ar[d,"\cong"]\\
    \Z/{2^{\min(r,t'+1)}}\langle \{ i_{n+1}\}\rangle\ar[r,"q_{n+1}^\ast",tail]&
     {[\m{r}{n},C]}\ar[r,"i_n^\ast",two heads] &K_2
  \end{tikzcd}\]
where $ K_1=\ker\big(\pi_n(Y)\xra{2^r}\pi_n(Y)\big),\quad 
K_2=\ker\big(\pi_n(C)\xra{2^r}\pi_n(C)\big)$.
 It follows that $g_\ast\colon [\m{r}{n}, Y]\to [\m{r}{n},C]$ is an epimorphism, where $[\m{r}{n},Y]$ has been known.
    Thus we compute that 
    \[[\m{r}{n},C^{n+2,t'}_{r'}]\cong \Z/{2^{\min(r,r')}}\langle i_{M'}B(\chi^r_{r'})\rangle\oplus \Z/{2^{\min(r,t'+1)}}
     \langle i_{n+1}q_{n+1} \rangle.\]

    For the group $[\m{t}{n+1},C^{n+2,t'}_{r'}]$, there exists an extension
    \[\ol{i_{M'}B(\chi^1_{r'})\tilde{\eta}}\colon \m{}{n+2}\to C^{n+2,t'}_{r'}\]
    such that $\ol{i_{M'}B(\chi^1_{r'})\tilde{\eta}}\circ i_{n+2}= i_{M'}B(\chi^1_{r'})\tilde{\eta}\colon S^{n+2}\to C^{n+2,t'}_{r'}$.
    One then checks that the induced homomorphism \[h=(\ol{i_{M'}B(\chi^1_{r'})\tilde{\eta}},i_{\oc'})_*: \pi_{j}(M_2^{n+2} \vee C^{n+2,t'})\to \pi_j(C^{n+2,t'}_{r'})\]
    is an isomorphism for $j=n+1,n+2$.
    Similarly, by the commutative diagram of exact sequences induced by $h$ and the five-lemma, we get an isomorphism 
    \[ [\m{t}{n+1},M_2^{n+2} \vee C^{n+2,t'}]\xra[\cong]{(\ol{i_{M'}B(\chi^1_{r'})\tilde{\eta}},i_{\oc})_\ast} [\m{t}{n+1},C^{n+2,t'}_{r'}] . \]
    Let  $n{''}=\min(t,t'+1)$, then
     \begin{align*}
      [\m{t}{n+1},C^{n+2,t'}_{r'}]\cong  & \Z/2\langle i_{M'}B(\chi^1_{r'})\tilde{\eta}q_{n+2}\rangle\oplus \Z/{2^{n{''}}}\langle i_{\oc'}\tilde{\xi}_{t'} B(\chi^t_{t'+1})\rangle \\
      &\oplus \Z/2\langle i_{n+1}\eta q_{n+2}\rangle.
     \end{align*}
    \end{proof}
    \end{proposition}

    \begin{lemma}\label{lem:xi}
      The following hold:
      \begin{enumerate}
        \item After choosing suitably, $\tilde{\xi}_t\in [\m{t+1}{n+1},C^{n+2,t}] $ simultaneously satisfies the formulas:
        \begin{equation}\label{eq:xi-def}
        \tilde{\xi}_t i_{n+1}= i_{n+1},\quad  q_\eta \tilde{\xi}_t = \tilde{\zeta}q_{n+2},\quad  q_{n+2}\tilde{\xi}_t = 2q_{n+2}, \quad  q_{M}\tilde{\xi}_t = B(\chi^{t+1}_t);
        \end{equation}
        \item Dually, there exist a map $\bar{\xi}_t\in [C^{n+2}_t,\m{t+1}{n}]$ simultaneously satisfies the formulas:
        \begin{equation}\label{eq:xi'-def}
        q_{n+1}\bar{\xi}_t = q_{n+1},\quad  \bar{\xi}_t i_\eta= i_n\bar{\zeta},\quad  \bar{\xi}_t i_n= 2 i_n,\quad  \bar{\xi} i_{M}= B(\chi^t_{t+1}).
        \end{equation}
      \end{enumerate}

    \begin{proof}
We only prove (\ref{eq:xi-def}) here and omit the similar proof of (\ref{eq:xi'-def}).  Consider the diagram with homotopy cofibre sequence rows:
\[\begin{tikzcd}
  S^{n+1}\ar[rr,"2^{t+1}"]\ar[d,"{\Sigma^{-1}\tilde{\zeta}}"]&&S^{n+1}\ar[r,"i_{n+1}"]\ar[d,equal]&\m{t+1}{n+1}\ar[r,"q_{n+2}"]\ar[d,dashed,"\tilde{\xi}_t"] &S^{n+2}\ar[d,"\tilde{\zeta}"]\\
    C^{n+1}_\eta\ar[rr,"2^t\Sigma^{-1}q_{n+1}"]&&S^{n+1}\ar[r,"i_{n+1}"]&C^{n+2,t}\ar[r,"q_\eta"]&\ceta
\end{tikzcd}\]
    By (\ref{eq:zeta'-def}), the first square is homotopy commutative, which implies that there exist a map  $\tilde{\xi}_t$ filling in the right two commutative squares: \[\tilde{\xi}_ti_{n+1}=i_{n+1},~q_\eta\tilde{\xi}_t=\tilde{\zeta}q_{n+2}.\]
   It then follows that 
    \[q_{n+2}\tilde{\xi}_t=q_{n+2}q_\eta \tilde{\xi}_t=q_{n+2}\tilde{\zeta}q_{n+2}=2q_{n+2}.\]
    For the last relation equality, by the group structure and generators of $[\m{t+1}{n+1},\m{t}{n+1}]$, we may put  \[q_{M_t}\tilde{\xi}_t=x\cdot B(\chi^{t+1}_t)+y\cdot i_{n+1}\eta q_{n+2}\] for some $x\in\zz{t},y\in\zz{}$. By composing $q_M$ on both sides of the equality from the left, together with Theorem \ref{thm:eqs} (\ref{eq:B}), we get
    \begin{align*}
    2q_{n+2}=q_{n+2}\tilde{\xi}_t&=q_{n+2}q_{M_t}\tilde{\xi}_t\\
    &=x\cdot q_{n+2}B(\chi^{t+1}_t)=2x\cdot q_{n+2}.
    \end{align*}
    Thus $x=1$.

    If  $y=0$, the proof is done; otherwise, substituting $\tilde{\xi}_t$ by $\tilde{\xi}_t+i_{n+1}\eta q_{n+2}$, then  $q_{M_t}\tilde{\xi}_t=B(\chi^{t+1}_t)$ holds. One can check that the new $\tilde{\xi}_t$ satisfies all the relation formulas discussed above, and therefore the proof is completed.
    \end{proof}
    \end{lemma}

Next we prove the group structure and generators in Table \ref{table2}.
By the \SW duality, it's easy to get the groups $[X,Y]$ with generators above the double lines, while the remaining entries of Table \ref{table2} are exactly determined by the following Proposition \ref{prop:CC} and \ref{prop:CC'}.

  \begin{proposition}\label{prop:CC}
    \begin{flushleft}
      
    \end{flushleft}
    \begin{enumerate}[$(1)$]
    \item $[\ceta,\ceta]\cong\Z\langle 1_\eta\rangle\oplus \Z\langle i_n\bar{\zeta}\rangle\cong\Z\langle 1_\eta\rangle\oplus \Z\langle\tilde{\zeta}q_{n+2}\rangle$, where $\bar{\zeta}\in [\ceta,S^n]$ and $\tilde{\zeta}\in [S^{n+2},\ceta]$ satisfy the formulas:
    \begin{equation}\label{eq:zeta-def}
    q_{n+2}\tilde{\zeta}=2\cdot 1_{n+2},~\bar{\zeta}i_n=2\cdot  1_n;~i_n\bar{\zeta}+\tilde{\zeta}q_{n+2}=2\cdot 1_\eta.
    \end{equation}

    \item $[\ceta,C^{n+2}_{r}]\cong \Z\langle i_\eta\rangle\oplus\zz{r}\langle i_n\bar{\zeta}\rangle $, $[\ceta,Y]\cong  \Z\langle i_n\bar{\zeta}\rangle$ for $Y=C^{n+2,t}$ or $C^{n+2,t}_r$.

    \item $[C^{n+2,t},C^{n+2,t'}]\cong \left\{\begin{array}{ll}
     \zz{t'+1}\langle \tilde{\xi}_{t'} B(\chi^t_{t'+1})q_M\rangle \oplus \Z\langle i_n\bar{\zeta}q_\eta \rangle&t> t';\\
     \zz{t}\langle \tilde{\xi}_{t'} B(\chi^t_{t'+1})q_M\rangle \oplus \Z\langle \tilde{\vartheta}^t_{t'} \rangle&t\leq t'
     \end{array}\right. $, where $\tilde{\vartheta}$ satisfies the relations ($t\leq t'$):
     \begin{equation}\label{eq:vartheta-def}
        \tilde{\vartheta}^t_{t'}i_n=i_n,~q_{M'}\tilde{\vartheta}^t_{t'}=B(\chi^t_{t'})q_M,~
        \tilde{\vartheta}^t_{t'}i_{n+1}=2^{t'-t}i_{n+1},~q_\eta\tilde{\vartheta}^t_{t'}=q_\eta.
     \end{equation}
    \begin{equation}\label{eq:2vartheta}
    i_n\bar{\zeta}q_\eta=2\tilde{\vartheta}^t_{t'}-\tilde{\xi}_{t'}B(\chi^t_{t'+1})q_M.
    \end{equation}

    \item\label{oc-uc} $[C^{n+2,t},C^{n+2}_{r'}]\cong \left\{
                         \begin{array}{ll}
                           \zz{t+1}\langle i_{\eta }q_\eta \rangle\oplus \zz{r'}\langle i_n\bar{\zeta}q_\eta \rangle, & \hbox{$t\geq r'$;} \\
                           \zz{r'+1}\langle i_{\eta }q_\eta \rangle\oplus \zz{t}\langle i_\eta\tilde{\zeta}q_{n+2}\rangle, & \hbox{$t<r'$;}
                         \end{array}\right.$.
    \item $[C^{n+2,t},C^{n+2,t'}_{r'}]\cong \left\{
    \begin{array}{ll}
     \zz{r'+1}\langle i_{\oc'}\tilde{\vartheta}^t_{t'} \rangle\oplus \zz{t} \langle i_{\oc'}\tilde{\xi}_{t'} B(\chi^t_{t'+1})q_M\rangle & \hbox{$t'\geq t<r'$;} \\
    \zz{t+1}\langle i_{\oc'}\tilde{\vartheta}^t_{t'} \rangle\oplus \zz{r'}\langle i_n\bar{\zeta}q_\eta  \rangle  &\hbox{$t'\geq t\geq r'$;}\\
    \zz{t'+1}\langle i_{\oc'}\tilde{\xi}_{t'} B(\chi^t_{t'+1})q_M\rangle\oplus \zz{r'} \langle i_n\bar{\zeta}q_\eta \rangle & \hbox{$t'<t$.}
    \end{array}\right.$.
                       
    \item $[C^{n+2}_r,C^{n+2,t'}]\cong\zz{\min(r,t')}\langle i_{n+1}q_{n+1}\rangle$.

    \item $[C^{n+2}_r,C^{n+2,t'}_{r'}]\cong \zz{\min(r+1,r')}\langle i_{M'}B(\chi^{r+1}_{r'})\bar{\xi}_{r} \rangle\oplus \zz{\min(r,t')+1}\langle i_{n+1}q_{n+1}\rangle$.

    \end{enumerate}
    \begin{proof}
    (1) Applying $[\ceta,-]$ to \textbf{Cof1} for $\ceta$, there is an exact sequence:
    \[0\to [\ceta,S^n]\xra{(i_n)_*}[\ceta,\ceta]\xra{(q_{n+2})_*}[\ceta,S^{n+2}]\to 0.\]
    Since $[\ceta,S^n]\cong \Z\langle \bar{\zeta}\rangle, [C^{n+2},S^{n+2}]\cong \Z\langle q_{n+2}\rangle$, where $\bar{\zeta}$ satisfies $\bar{\zeta}i_n=2\cdot 1_n$,   the above exact sequence  splits. Hence the group $[\ceta,\ceta]$ is proved. The other generating set follows from the formula (cf. \citep[Section 8.1]{Toda2}):
    \[i_n\bar{\zeta}+\tilde{\zeta}q_{n+2}=2\cdot 1_\eta.\]

    (2) The group structure and generators are immediate by applying the exact functor $[\ceta,-]$ to \textbf{Cof1}s for $C^{n+2,t},C^{n+2}_{r},C^{n+2,t}_{r}$, respectively.  

    (3) Write $\oc'=C^{n+2,t'}$ for short. There is an exact sequence by applying $[-,C^{n+2,t'}]$ to \textbf{Cof2} for $C^{n+2,t}$:
  \[[S^{n+1},\oc']\xra{(\eta q_{n+2})^*}[\m{t}{n+1},\oc']\xra{q_M^*}[C^{n+2,t},\oc']\xra{i_n^*}[S^{n},\oc']\xra{(\eta q_{n+1})^*}[\m{t}{n},\oc'],
  \]
 where all groups except $[C^{n+2,t},\oc']$ are listed in Table \ref{table1}. Since $i_n\eta=2^{t'}\cdot i_{n+1}\in [S^{n+1},\oc']$ (Lemma \ref{lem:cofib}), we get the following two splitting short exact sequences:
    \begin{align*}
    t>t':&\quad 0\to\zz{t'+1}\langle \tilde{\xi}_{t'} B(\chi^t_{t'+1})\rangle\xra{q_M^*}[C^{n+2,t},C^{n+2,t'}]\xra{i_n^*}\Z\langle 2\cdot i_n\rangle\to 0;\\
    t\leq t':&\quad 0\to\zz{t}\langle \tilde{\xi}_{t'} B(\chi^t_{t'+1})\rangle\xra{q_M^*}[C^{n+2,t},C^{n+2,t'}]\xra{i_n^*}\Z\langle i_n\rangle\to 0.
    \end{align*}
    If $t>t'$, note that $i_n^\ast(i_n\bar{\zeta}q_\eta)=2\cdot i_n$, we have \[[C^{n+2,t},C^{n+2,t'}]\cong\zz{t'+1}\langle  \tilde{\xi}_{t'} B(\chi^t_{t'+1})q_M\rangle \oplus \Z\langle i_n\bar{\zeta}q_\eta \rangle.\]
    If  $t\leq t'$, let $\tilde{\vartheta}^t_{t'}\in [C^{n+2,t},C^{n+2,t'}]$ satisfies \[
    \tilde{\vartheta}^t_{t'}i_n=i_n,~\tilde{\vartheta}^t_t=1_{\oc},\]
    then we get $[C^{n+2,t},C^{n+2,t'}]\cong\zz{t}\langle  \tilde{\xi}_{t'} B(\chi^t_{t'+1})q_M\rangle \oplus \Z\langle \tilde{\vartheta}^t_{t'} \rangle.$

    For the relation formulas (\ref{eq:vartheta-def}), (\ref{eq:2vartheta}), consider the following commutative diagram with homotopy cofibre sequence rows:
    \[\begin{tikzcd}
      \m{t}{n}\ar[d,"B(\chi^t_{t'})"]\ar[r,"\eta q_{n+1}"]&S^n\ar[r,"i_n"]\ar[d,equal]&C^{n+2,t}\ar[r,"q_M"]\ar[d,"\tilde{\vartheta}^t_{t'}",dotted]&\m{t}{n+1}\ar[d,"B(\chi^t_{t'})"]\\
    \m{t'}{n}\ar[r,"\eta q_{n+1}'"]& S^n\ar[r,"i_n"]&C^{n+2,t'}\ar[r,"q_{M'}"]&\m{t'}{n+1}
    \end{tikzcd}\]
  Hence the map $\tilde{\vartheta}^t_{t'}$ satisfies $q_{M'}\tilde{\vartheta}^t_{t'}=B(\chi^t_{t'})q_M.$
 The equality $\tilde{\vartheta}^t_{t'}i_{n+1}=2^{t'-t}i_{n+1}$ then follows, since 
 \[q_{M}\circ i_{n+1}=i_{n+1}\colon S^{n+1}\xra{i_{n+1}} C^{n+2,t}\xra{q_M} \m{t}{n+1}.\]

  Since $[C^{n+2,t},\ceta]\cong\Z\langle q_\eta\rangle\oplus\zz{t}\langle\tilde{\zeta}q_{n+2}\rangle$, we can set  
  \[q_\eta\tilde{\vartheta}^t_{t'}=x\cdot q_\eta+y\cdot \tilde{\zeta}q_{n+2}\]
  for some $x\in\Z,y\in\zz{t}$.
    By composing $q_{n+2}$ on both sides from the left, we have
    \begin{align*}
    (x+2y)\cdot q_{n+2}q_{n+2}\tilde{\vartheta}^t_{t'}&=q_{n+2}\tilde{\vartheta}^t_{t'}=q_{n+2}q_{M'}\tilde{\vartheta}^t_{t'}=q_{n+2}B(\chi^t_{t'})q_M=q_{n+2},
    \end{align*}
    hence  $x+2y=1$.  Composing $i_n$  on both sides from the right, we have  $i_n=x\cdot i_n$ and hence  $x=1,y=0$. Thus $q_\eta \tilde{\vartheta}^t_{t'}=q_\eta$ is proved.

    Since $i_n\bar{\zeta}q_\eta i_n= i_n\bar{\zeta}i_n= 2i_n\in[S^n, C^{n+2,t'}]$, $\tilde{\vartheta}^t_{t'} i_n= i_n$,  we may put  
    \[i_n\bar{\zeta}q_\eta =x\cdot \tilde{\xi}_{t'} B(\chi^t_{t'+1})q_M+2\cdot \tilde{\vartheta}^t_{t'}  ~~,x\in\zz{t}.\]
    By composing $q_{\eta}$ on both sides from the left, we have
    \begin{align*}
     2\cdot q_\eta -\tilde{\zeta}q_{n+2}&= q_\eta i_n\bar{\zeta}q_\eta = i_n\bar{\zeta}q_\eta  ,\text{by ( \ref{eq:zeta-def})}\\
     &= x\cdot q_\eta \tilde{\xi}_{t'} B(\chi^t_{t'+1})q_M+2\cdot q_\eta \tilde{\vartheta}^t_{t'} \\
     &=x\cdot \tilde{\zeta}q_{n+2}B(\chi^t_{t'+1})q_M+2\cdot q_\eta \tilde{\vartheta}^t_{t'} ,\text{by (\ref{eq:xi-def})}.
     \end{align*}
     Note that the composition $q_{n+2}B(\chi^t_{t'+1})q_M$ is homotopic to $q_{n+2}$:
     \[q_{n+2}\colon C^{n+2,t}\xra{q_M}\m{t}{n+1}\xra{B(\chi^t_{t'+1})}\m{t'+1}{n+1}\xra{q_{n+2}}S^{n+2}.\]
     It follows that $x=-1$, and hence (\ref{eq:2vartheta}) is proved.

    (4) If $t\geq r'$, consider the exact sequence induced by \textbf{Cof3} for $C^{n+2,t}$:
    \[[S^{n+2},C^{n+2}_{r'}]\xra{(2^tq_{n+2})^*}[\ceta,C^{n+2}_{r'}]\xra{q_\eta ^*}[C^{n+2,t},C^{n+2}_{r'}]\to 0.\]
    By the generators of the first two groups and relations in (\ref{eq:zeta-def}), we have \[2^t\cdot i_\eta\tilde{\zeta}q_{n+2}=2^{t+1}\cdot i_\eta.\] 
    Thus $[C^{n+2,t'},C^{n+2}_{r'}]\cong \zz{t+1}\langle i_{\eta }q_\eta \rangle\oplus \zz{r'}\langle i_n\bar{\zeta}q_\eta \rangle.$

    If  $t<r'$, consider the exact sequence induced by \textbf{Cof3} for $C^{n+2}_{r'}$:
    \[[C^{n+2,t},S^n]\xra{i_n2^{r'}}[C^{n+2,t},\ceta]\xra{i_\eta}[C^{n+2,t},C^{n+2}_{r'}]\to 0.  \]
    hence $2^{r'}\cdot i_n\bar{\zeta}q_\eta = 2^{r'+1}\cdot q_\eta$, and therefore \[[C^{n+2,t},C^{n+2}_{r'}]\cong \zz{r'+1}\langle i_{\eta}q_\eta \rangle\oplus \zz{t}\langle i_\eta\tilde{\zeta}q_{n+2}\rangle.\]

    (5) There is an exact sequence induced by \textbf{Cof6} for  $C^{n+2,t'}_{r'}$: 
    \[[C^{n+2,t},S^{n}]\xra{(i_n2^{r'})_*}[C^{n+2,t},C^{n+2,t'}]\xra{(i_{\oc'})_*}[C^{n+2,t},C^{n+2,t'}_{r'}]\to 0.\leqno (*)\]

   If $t'< t$,  the above exact sequence turns to be \[\Z\langle \bar{\zeta}q_\eta \rangle\xra{(i_n2^{r'})_\ast}\zz{t'+1}
    \langle  \tilde{\xi}_{t'} B(\chi^t_{t'+1})q_M\rangle\oplus\Z\langle i_n\bar{\zeta}q_\eta \rangle\xra{(i_{\oc'})_\ast}[C^{n+2,t},C^{n+2,t'}_{r'}]\to 0.\]
  Hence in this case we have 
  \[ [C^{n+2,t},C^{n+2,t'}_{r'}]\cong \zz{t'+1}\langle i_{\oc'} \tilde{\xi}_{t'} B(\chi^t_{t'+1})q_M\rangle\oplus \zz{r'} \langle i_n\bar{\zeta}q_\eta \rangle.\]

  If $t'\geq t$,  applying $[C^{n+2,t},-]$ to \textbf{Cof5} for $C^{n+2,t'}_{r'}$, we have an exact sequence
    \[0\to [C^{n+2,t},C^{n+2,t'}_{r'}]\xra[\cong]{(q_{\uc'})_*}[C^{n+2,t},C^{n+2}_{r'}]\xra{(2^{t'}q_{n+2})_*=0}[C^{n+2,t},S^{n+2}]\cong \zz{t}.\]
    Then the exact sequence $(*)$ turns to be
    \[\Z\langle \bar{\zeta}q_\eta \rangle\xra{(i_n2^{r'})_*}\zz{t}
    \langle \tilde{\xi}_{t'} B(\chi^t_{t'+1})q_M\rangle\oplus\Z\langle \tilde{\vartheta}^t_{t'} \rangle\xra{(i_{\oc'})_*}[C^{n+2,t},C^{n+2,t'}_{r'}]\to 0.\leqno(\ast')\]
     If  $r'\geq t\leq t'$, then  \[[C^{n+2,t},C^{n+2,t'}_{r'}]\cong\zz{t}\langle i_{\oc'}\tilde{\xi}_{t'} B(\chi^t_{t'+1})q_M\rangle\oplus \zz{r'+1}\langle i_{\oc'}\tilde{\vartheta}^t_{t'} \rangle.\]
    If $r'\leq t\leq t'$, then \[[C^{n+2,t},C^{n+2,t'}_{r'}]\xleftarrow[\cong]{(i_{\oc})_\ast}\frac{\zz{t}\langle \tilde{\xi}_{t'} B(\chi^t_{t'+1})q_M\rangle\oplus \Z\langle {\tilde{\vartheta}^t_{t'} }\rangle}{\langle 2^{r'}(-\tilde{\xi}_{t'} B(\chi^t_{t'})q_M,2{\tilde{\vartheta}^t_{t'} })\rangle}\cong \zz{t+1}\langle\{\tilde{\vartheta}^t_{t'} \}\rangle\oplus\zz{r'}\langle X\rangle,\]
    where $X=(-\tilde{\xi}_{t'} B(\chi^t_{t'+1})q_M,2\tilde{\vartheta}^t_{t'} )= i_n\bar{\zeta}q_\eta $, by (\ref{eq:2vartheta}).

    (6)  Consider the exact sequence induced by \textbf{Cof1} for $C^{n+2}_r$:
    \[\pi_{n+1}(C^{n+2,t'})\xra{\binom{2^r}{\eta}^\ast}\pi_{n+1}(C^{n+2,t'})\oplus \pi_{n+2}(C^{n+2,t'})\xra{q_S^\ast}[C^{n+2}_r,C^{n+2,t'}]\to 0,\]
where $q_S=\binom{q_{n+1}}{q_{n+2}}$. We compute that \[\coker\binom{2^r}{\eta}^*=\frac{\zz{t'+1}\langle i_{n+1}\rangle\oplus \Z/2\langle i_{n+1}\eta\rangle}{\langle(2^ri_{n+1},i_{n+1}\eta)\rangle}\cong\zz{\min(r,t')+1}\langle \{i_{n+1}\}\rangle.\]

    (7) Write $Y=M_{2^{r'}}^{n} \vee C^{n+2,t'}$ and $C'=C^{n+2,t'}_{r'}$. There is a commutative diagram with exact rows induced by $g=({i_{M'}}, i_{\oc'})\colon Y\to C'$:
    \[\begin{tikzcd}[sep=scriptsize]
     \pi_{n+1}(Y)\ar[r,"\binom{2^r}{\eta}^\ast"]\ar[d,"g_\ast"]&G_2(Y)\ar[r,"q_S^\ast"]\ar[d,"g_\ast"]&{[C^{n+2}_r,Y]}\ar[r,"i_n^\ast"]\ar[d,"g_\ast"]&\pi_n(Y)\ar[r,"\binom{2^r}{\eta}^\ast"]\ar[d,"g_\ast"]&G_1(Y)\ar[d,"g_\ast"]\\
     \pi_{n+1}(C')\ar[r,"\binom{2^r}{\eta}^\ast"]&G_2(C')\ar[r,"q_S^\ast"]&{[C^{n+2}_r,C']}\ar[r,"i_n^\ast"]&\pi_n(C')\ar[r,"\binom{2^r}{\eta}^\ast"]&G_1(C')
    \end{tikzcd}\]
    where $G_i(Z)=\pi_{n+i-1}(Z)\oplus\pi_{n+i}(Z)$, $Z=Y,C',i=1,2$.
   After computing the cokernels and kernels of the homomorphisms $\binom{2^r}{\eta}^\ast$, we get the commutative diagram of exact sequences:
   \[\begin{tikzcd}[row sep=scriptsize,column sep=tiny]
    \Z/2\langle i_n\eta\rangle\ar[d,tail]\ar[r,"q_{n+1}^\ast",tail] & \Z/2\langle i_n\eta q_{n+1}\rangle \ar[d,tail]\ar[r]&0\ar[d]\\
    \zz{m+1}\langle \{i_{n+1}\} \rangle\oplus\Z/2\langle \{B(\chi^1_{r'})\tilde{\eta}\}\rangle\oplus \Z/2\langle i_n\eta \rangle\ar[r,tail]\ar[d,"\widetilde{g_\ast}",two heads]&
     {[C^{n+2}_r,Y]}\ar[r,two heads]\ar[d,"{(i_{M'},i_{\oc'})}_\ast",two heads] & K_1\ar[d,"\cong"]\\
     \zz{m+1}\langle \{i_{n+1}\} \rangle\oplus\Z/2\langle i_{M'}B(\chi^1_{r'})\tilde{\eta}\rangle\ar[r,"q_S^\ast",tail]
     & {[C^{n+2}_r,C']}\ar[r,two heads] & K_2
   \end{tikzcd}\]
    where the two groups $K_1,K_2$ are the kernels of the homomorphisms \[\zz{r'}\langle i_n\rangle\xra{(2^r,\eta)^*}\zz{r'}\langle i_n\rangle\oplus \zz{t'+1}\langle i_{n+1}\rangle.\]
    Then by the exactness of the middle short exact sequence and the group
    \[[C^{n+2}_r,Y]\cong\zz{l'}\langle B(\chi^{r+1}_{r'})\bar{\xi}_r \rangle\oplus\zz{m+1}\langle i_{n+1}q_{n+1}\rangle\oplus \Z/2\langle i_n\eta q_{n+1}\rangle,\]
 we get $[C^{n+2}_r,C']\cong \zz{l'}\langle i_{M'}B(\chi^{r+1}_{r'})\bar{\xi}_r \rangle\oplus \zz{m+1}\langle i_{n+1}q_{n+1}\rangle$,
    where $l'=\min(r+1,r'),m=\min(r,t')$.
    \end{proof}
    \end{proposition}

\begin{lemma}\label{lem:eqqi=iq}
  There holds $q_{\uc}i_{\oc} = i_\eta q_\eta$; i.e., there is a homotopy commutative square
  \[\begin{tikzcd}[sep=scriptsize]
    C^{n+2,t}\ar[r,"i_{\oc}"]\ar[d,"q_\eta"]&\CC\ar[d,"q_{\uc}"]\\
    \ceta\ar[r,"i_\eta"]&C^{n+2}_r
  \end{tikzcd}\]
\begin{proof}
For simplicity denote $\beta=i_n\bar{\zeta}q_\eta$ if $t\geq r$; otherwise $\beta=i_\eta\tilde{\zeta}q_{n+2}$.   By Proposition \ref{prop:CC} (\ref{oc-uc}) we may put  
\[q_{\uc}i_{\oc} =x\cdot i_\eta q_\eta+y\cdot \beta\]
for some $x\in\zz{\max(r,t)+1},y\in\zz{\min(r,t)}$. 

If $t\geq r$, by composing $i_n$ on both sides from the right, we have 
   \[i_n= (x+2y)\cdot i_n.\]
   By composing  $\tilde{\zeta}q_{n+2}$ from the left, we have  
   \[\tilde{\zeta}q_{n+2}= x\cdot \tilde{\zeta}q_{n+2}+0.\] 
   Thus $x=1,y=0$, and therefore $q_{\uc}i_{\oc}=i_\eta q_\eta$ in this case.

The proof of the formula in the case $t<r$ is similar.
\end{proof}
\end{lemma}   

The last group we need to compute is $[C^{n+2,t}_{r},C^{n+2,t'}_{r'}]$, considering the page layout we give an alternative method to get its generators in terms of its group structure given by \cite{Baues85}.

\begin{lemma}\label{lem:ES}
    Let $A,B,C$ be finitely generated abelian groups. Then the short exact sequence
      \begin{equation}\label{ES}
        0\to A\xra{f}B\xra{g} C\to 0.\tag{$\diamond$}
      \end{equation}
      splits if and only if $B\cong A\oplus C$.

      \begin{proof}
        The ``only if" part is clear. For the ``if" part, recall that the exact sequence (\ref{ES}) splits if and only if the sequence
        \begin{equation}
          0\to \Hom(C,A)\xra{g^*} \Hom (B,A)\xra{f^*} \Hom(A,A)\to 0.\tag{$\diamond'$}\label{ES:hom}
        \end{equation} 
        is exact; or equivalently, the homomorphism $f^\ast$ is an epimorphism.

        Firstly we assume that $A,B,C$ are $\Z$-modules of finite lengths.
        Let $l(M)$ denote the length of an $R$-module $M$. For the exact sequence  (\ref{ES}), there holds $l(B)=l(A)+l(C)$. Then
       \begin{align*}
       l(\Hom(B,A))&=l(\Hom(C,A))+l(\Hom(A,A));\\
       &=  l(\Hom(C,A))+l(\im(f^*))
       \end{align*}
      It follows that $l(\im(f^*))=l(\Hom(A,A))$, $\im(f^*)=\Hom(A,A)$, and hence
     the exact sequence  (\ref{ES}) splits.

    For general finitely abelian groups, note that the sequence (\ref{ES:hom}) is exact if and only if it is exact after localizing at arbitrary prime $p$. Note also that the ring $\Z_p$ of $p$-adic integers are \textit{faithfully flat}, which means it is flat and the tensor functor $-\otimes \Z_p\colon \mathsf{Ab}\to\mathsf{Ab}$ reflects exactness. It follows that the $p$-localized sequence of (\ref{ES:hom}) is exact if and only if it is exact after $p$-completion. Since $\Z_p$ is the inverse limit $\varprojlim \Z/p^n$, and the functor $\Hom(-,-)$ commutes with flat base change, it suffices to show the exactness of the $\zp{n}$-tensored sequence 
        \[
         \begin{tikzcd}[column sep=small]
            \Hom(C,A)\otimes\zp{n}\ar[r,tail] &\Hom (B,A)\otimes\zp{n}\ar[r,two heads]& \Hom(A,A)\otimes\zp{n}.
          \end{tikzcd}
        \] 
   For this we may assume that the lengths are finite, which we have dealt with.
 \end{proof}
\end{lemma}

Let $j=\max(t,r'),k=\min(t,r'),m=\min(r,t'),n''=\min(t,t'+1),l'=\min(r+1,r')$. From \cite{Baues85} we have 
\begin{equation}
  [C^{n+2,t}_{r},C^{n+2,t'}_{r'}]\cong \zz{m+1}\oplus \left\{\begin{array}{ll}
   \zz{j+1}\oplus\zz{k},&r'\leq r\wedge t'\geq t;\\
   \zz{n''}\oplus\zz{l'},&r'>r\vee t'<t.
 \end{array}\right. \tag{$\star$}\label{grp:CC}
 \end{equation}

\begin{proposition}\label{prop:CC'}
Let $j,m,n'',l'$ be as above and let $G=[\CC,C^{n+2,t'}_{r'}]$. 
\begin{enumerate}
  \item If $r'>r$ or $ t'<t$,
   \[G\cong \zz{m+1}\langle i_{n+1}q_{n+1}\rangle\oplus \zz{l'}\langle i_{M'}B(\chi^{r+1}_{r'})\bar{\xi}_{r}q_{\uc}\rangle\oplus \zz{n''}\langle i_{\oc'}\tilde{\xi}_{t'}B(\chi^t_{t'+1})q_M\rangle.\]
  \item If $r'\leq r$ and $t\leq t'$, 
 \[ G\cong \zz{m+1}\langle i_{n+1}q_{n+1}\rangle\oplus\zz{j+1}\langle L(\chi)\rangle \oplus \zz{k}\langle \omega^t_{r'}\rangle,\]
 where $\omega^t_{r'}=\left\{\begin{array}{ll}
   i_{M'}B(\chi^{r+1}_{r'})\bar{\xi}_r q_{\uc}&\text{ if }t\geq r'\\
   i_{\oc'}\tilde{\xi}_{t'} B(\chi^t_{t'+1})q_M&\text{ if }t<r'
 \end{array}\right.$,
 $L(\chi)$ simultaneously satisfies the formulas:
    \begin{gather}\label{eq:CC-def}
      \begin{aligned}
        L(\chi) i_{\oc}=i_{\oc'}\tilde{\vartheta}^t_{t'},\quad& q_{\uc'}L(\chi)=\bar{\vartheta}q_{\uc};\\
     L(\chi)i_M=i_{M'}B(\chi^r_{r'}),&\quad q_{M'}L(\chi)=B(\chi^t_{t'})q_M.  
     \end{aligned}
    \end{gather}
\end{enumerate}

  \begin{proof}
 Write $C=\CC,C'=C^{n+2,t'}_{r'}$ for short. 
 Applying Lemma \ref{lem:ES} to the group structure (\ref{grp:CC}), we have the following splitting short exact sequences:
   \begin{align*}
    t'<t:&\quad 0\to [C^{n+2}_r,C']\xra{q_{\uc}^*}[C,C']\xra{i_{n+1}^*}\zz{t'+1}\langle i_{n+1}\rangle\to 0;\\
     t'\geq t:&\quad 0\to \zz{m+1}\langle i_{n+1}\rangle\xra{q_{n+1}^*}[C,C']\xra{i_{\oc}^*}[C^{n+2,t},C']\to 0.
   \end{align*} 

(1) If $t'<t$, $[C^{n+2}_r,C']\cong \zz{l'}\langle i_{M'}B(\chi^{r+1}_{r'})\bar{\xi}_r \rangle\oplus \zz{m+1}\langle i_{n+1}q_{n+1}\rangle$, it suffices to choose a generator of the direct summand $\zz{t'+1}$ of $[C,C']$. 
By (\ref{eq:xi-def}), we check that 
\[i_{\oc'}\tilde{\xi}_{t'}B(\chi^t_{t'+1})q_M\circ i_{n+1}=i_{n+1},\] 
hence $i_{\oc'}\tilde{\xi}_{t'}B(\chi^t_{t'+1})q_M$ is a generator of $\zz{t'+1}$.

(2) If $t'\geq t$, it is clear that $i_{n+1}q_{n+1}$ is a generator of the direct sumand $\zz{m+1}$. Recall that 
\[
    [C^{n+2,t},C']\cong \zz{j+1}\langle i_{\oc'}\tilde{\vartheta}^t_{t'} \rangle\oplus\left\{
      \begin{array}{ll}
         \zz{t} \langle i_{\oc'}\tilde{\xi}_{t'} B(\chi^t_{t'+1})q_M\rangle & \hbox{$t'\geq t<r'$;} \\
         \zz{r'}\langle i_n\bar{\zeta}q_\eta  \rangle  &\hbox{$t'\geq t\geq r'$;}
      \end{array}\right.\]
Define $L(\chi)\in [C,C']$ by the equation
\[L(\chi)\circ i_{\oc}=i_{\oc'}\tilde{\vartheta}^t_{t'}.\]
Then $L(\chi)$ is a generator of the direct summand $\zz{j+1}$ of $[C,C']$.
By Lemma \ref{lem:eqqi=iq} and the dual formulas of (\ref{eq:xi-def}), 
  \[i_{M'}B(\chi^{r+1}_{r'})\bar{\xi}_{r} q_{\uc}\circ i_{\oc}= i_{M'}B(\chi^{r+1}_{r'})\bar{\xi}_{r} i_\eta q_\eta= i_{M'}B(\chi^{r+1}_{r'})i_n\bar{\zeta}q_\eta = i_n\bar{\zeta}q_\eta.\]
Thus $i_{M'}B(\chi^{r+1}_{r'})\bar{\xi}_{r} q_{\uc}\in [C,C']$ is a generator of the direct summand the direct summand $\zz{k}=\zz{r'}$.

Now we prove other formulas of (\ref{eq:CC-def}).
By the group $[\m{r}{n},C']$ we may put 
\[L(\chi) i_M=x\cdot i_{M'}B(\chi^r_{r'})+y\cdot i_{n+1}q_{n+1}\] for some $x\in\zz{r'},y\in\zz{\min(r,t'+1)}$.
 Note that $L(\chi)i_Mi_n= \tilde{\vartheta}^t_{t'} i_n= i_n$, we get  $x=1$, and hence 
 \begin{equation}\label{eq:CC-iM}
   L(\chi)i_M=i_{M'}B(\chi^r_{r'})+y\cdot i_{n+1}q_{n+1}.
 \end{equation}
By the group $[C,\m{t'}{n+1}]$ we may put 
\[q_{M'}L(\chi)=x'\cdot B(\chi^{t}_{t'})q_M+y'\cdot i_{n+1}q_{n+1} \]
for some $x'\in\zz{n},y'\in\zz{m'}$.
By composing $i_{\oc}$ from the right, and by (\ref{eq:vartheta-def}), we have 
\begin{align*}
  B(\chi^t_{t'})q_M=q_{M'}\tilde{\vartheta}^t_{t'}&=q_{M'}\circ i_{\oc'}\tilde{\vartheta}^t_{t'}
  =q_{M'}\circ L(\chi) i_{\oc}\\
  &=x'\cdot B(\chi^t_{t'})q_M\circ i_{\oc}=x'\cdot B(\chi^t_{t'})q_M.
\end{align*}
Hence $x'=1$, and $q_{M'}L(\chi)= B(\chi^t_{t'})q_M+y'\cdot i_{n+1}q_{n+1}.$

Composing $i_M$ from the right, and by (\ref{eq:CC-iM}), we get $y=y'$. 
If $y\neq 0$, substituting $L(\chi)$ by $L(\chi)-y\cdot i_{n+1}q_{n+1}$, then we complete the proofs of the two formulas
\[L(\chi)i_M=i_{M'}B(\chi^r_{r'}),\quad q_{M'}L(\chi)= B(\chi^{t}_{t'})q_M.\]

By the group $[C,C^{n+2}_{r'}]$, there exist $x''\in\zz{j+1},y''\in\zz{k}$ such that  
\[q_{\uc'}L(\chi)=x''\cdot \bar{\vartheta}q_{\uc}+y''\cdot \left\{\begin{array}{ll}
  i_{M'}B(\chi^{r+1}_{r'})\bar{\xi}_{r}q_{\uc}, &t\geq r'\leq r;\\
  i_\eta\tilde{\zeta}q_{n+2},&t\geq r'>r.
\end{array} \right.\]
Composing $i_{\oc'}$ from the right and by (\ref{eq:vartheta}), we have 
\[
  i_\eta q_\eta=x''\cdot i_\eta q_\eta+y''\cdot \left\{\begin{array}{ll}
    i_n\bar{\zeta}q_\eta,&t\geq r'\leq r;\\
    i_\eta \tilde{\zeta}q_{n+2},&t\geq r'>r.
  \end{array} \right.
\]
Note that these composites are generators of $[C^{n+2,t},C^{n+2}_{r'}]$, we have $x''=1,y''=0$, and hence $q_{\uc'}L(\chi)=\bar{\vartheta}^r_{r'}q_{\uc}$ holds.
  \end{proof}    
\end{proposition}

\begin{proof}[Proof of Theorem \ref{thm:eqs}]\label{proof:thm-eqs}
The items $(3)\sim (7)$ of Theorem \ref{thm:eqs} summarize the relation formulas (\ref{eq:xi-def})$\sim$(\ref{eq:vartheta-def}), (\ref{eq:CC-def}).
\end{proof}

\section{Self-homotopy equivalences of Chang complexes}\label{sec:htpEq}
In this section we prove the theorems and corollaries listed in Section \ref{sec:intro}. The following lemma is well-known. 
\begin{lemma}[cf. \cite{ZP17}]\label{lem:Sq2}
  For every (indecomposable) Chang complex $X$, the Steenrod square \[\sq^2\colon H^n(X;\zz{})\to H^{n+2}(X;\zz{})\] is an isomorphism.
\end{lemma}

Let $\sigma_0=1\in H_0(S^0)$ be a fixed generator of the reduced $0$-dimensional homology group and let $\sigma_n\in H_n(S^n)$ be the image of $\sigma_0$ under the isomorphism $\Sigma^{n}\colon H_0(S^0)\xra{\cong}H_n(S^n)$. Denote by $\delta_{ij}$ the Kronecker delta. The homology groups of $\CC$ are given by  
\begin{equation}\label{eq:homology}
  H_i(\CC)\cong \left\{\begin{array}{ll}
  \zz{r}\langle (i_n)_\ast\sigma_n\rangle,&i=n;\\
  \zz{t}\langle (i_{n+1})_\ast\sigma_{n+1}\rangle,&i=n+1\\
  0&\text{otherwise}.
  \end{array} \right.
\end{equation}

\subsection{Proof of Theorem \ref{thm:ECC}}\label{sec:ECC}

Denote $\omega^t_r=i_{M}B(\chi^{r+1}_{r})\bar{\xi}_rq_{\uc}$ if $t\geq r$, otherwise $\omega^t_r=
i_{\oc}\tilde{\xi}_tB(\chi^t_{t+1})q_M$. Let $C=\CC$ and let $j=\max(r,t)$, $m=\min(r,t)$. 
Recall that \[ [C,C]\cong \zz{j+1}\langle 1_C\rangle\oplus\zz{m+1}\langle i_{n+1}q_{n+1}\rangle \oplus\zz{m}\langle \omega^t_r\rangle,\] 
 for simplicity we denote a self-map $f$ of $C$ coordinately by   
 \[f=(x,y,z)=x\cdot 1_C+y\cdot i_{n+1}q_{n+1}+z\cdot \omega^t_r,\]
 where $x\in \zz{j+1},y\in\zz{m+1},z\in\zz{m}$. 

\begin{lemma}\label{lem:ECC-eqs}
  Using the above notation, there hold formulas: 
  \begin{enumerate}
    \item\label{eq:t>r} If $t\geq r$,  
    \begin{align*}
      (x,y,z)\circ i_n=(x+2z)\cdot i_n,&\quad (x,y,z)\circ i_{n+1}=x\cdot i_{n+1},\\
      q_{n+1}\circ  (x,y,z)&=(x+2z)\cdot q_{n+1},\\
       (x,y,z)\circ (i_MB(\chi^1_r)\tilde{\eta})&=(x+2z)\cdot i_MB(\chi^1_r)\tilde{\eta}+y\cdot i_{n+1}\eta.
    \end{align*}
    \item\label{eq:t<r} If $t<r$, 
    \begin{align*}
      (x,y,z)\circ i_n=x\cdot i_n,&\quad (x,y,z)\circ i_{n+1}=(x+2z)\cdot i_{n+1},\\
      q_{n+1}\circ (x,y,z)&=x\cdot q_{n+1},\\
      (x,y,z)\circ (i_MB(\chi^1_r)\tilde{\eta})&=x\cdot i_MB(\chi^1_r)\tilde{\eta}+y\cdot i_{n+1}\eta
    \end{align*}
  \end{enumerate} 
\begin{proof}
  Direct applications of relation formulas in Theorem \ref{thm:eqs}.
\end{proof}
\end{lemma}

\begin{lemma}\label{lem:ECC}
As a set, \[\E(C)=\left\{(x,y,z)\in \zz{j+1}\oplus\zz{m+1}\oplus\zz{m}:x\equiv 1\pmod 2\right\}.\]
In particular, $\E(C)$ has order $2^{t+r+\min(r,t)+1}$.

\begin{proof}
By \citep[Lemma 3.1]{ZLP}, $f=(x,y,z)\in\E(C)$ if and only if $f$ induces an automorphism on $H_n(C;\zz{})$. 
Using $\zz{}$ coefficients, Lemma \ref{lem:ECC-eqs} implies that 
\[(x,y,z)_\ast\big((i_n)_\ast\sigma_n\big)=x\cdot (i_n)_\ast\sigma_n.\] 
Hence $f=(x,y,z)\in\E(C)$ if and only if $x\equiv 1\pmod 2$. 

The order of $\E(C)$ follows immediately.
\end{proof}
\end{lemma}

\begin{proof}[Proof of Theorem \ref{thm:ECC}]
Write $C=\CC$ for short. It suffices to show that the above map $\pi$ is an epimorphism and admits a section.
Given a self-map $f$ of $C$, the induced endomorphism $\pi^{n+1}(f)$ on $  \pi^{n+1}(C)$ is a multiplication by certain integer $k_f$. It follws that the map 
\[\phi\colon \E(C)\to \Aut(\pi_{n+1}(C))\oplus\Aut(\pi^{n+1}(C))\] defined in the theorem is a homomorphism of groups. 

 By Theorem \ref{thm:eqs} we have the composition laws in $[C,C]$: 
 \begin{equation*}
  (x,y,z)\circ (x',y',z')=\left\{\begin{array}{ll}
    (xx',xy'+x'y+2yz',xz'+x'z+2zz'),&t\geq r;\\[1ex]
    (xx',xy'+x'y+2y'z,xz'+x'z+2zz'),&t<r.
  \end{array}\right.
 \end{equation*}
Denote by $\rho_k\colon \Z/l\to\Z/l$ the multiplication by $k$ on $\Z/l$, $l\geq 2$. By Lemma \ref{lem:ECC-eqs}, with the notation $f=(x,y,z)$, $\phi$ can be expressed as
\begin{equation*}
  \phi(x,y,z)=\left\{\begin{array}{ll}
  (\rho_x, \rho_{x+2z}),&\text{ if } t\geq r;\\
  (\rho_{x+2z},\rho_x),&\text{ if } t<r.
\end{array}\right.
\end{equation*} 
Then in both cases we have 
\begin{align*}
  \ker(\phi)&=\{(1,y,0)\in\E(C):y\in\zz{\min(r,t)+1}\}\\
  &=\zz{\min(r,t)+1}\langle(1,1,0)\rangle.
\end{align*}
Write $Q=\Aut(\pi_{n+1}(C))\oplus\Aut(\pi^{n+1}(C))$ for simpliciy. By Lemma \ref{lem:ECC}, the image subgroup $\phi(\E(C))$ has order $2^{t+r}$, which is also the order of $Q$. It follows that the homomorphism $\phi$ is surjective.
Define a map $\iota\colon Q\to \E(C)$ by 
\begin{align*}
  t\geq r:{}&\iota(\rho_x,\rho_{x+2z})=(x,0,z);\\
  t<r:{}&\iota(\rho_{x+2z},\rho_{x})=(x,0,z).
\end{align*}
It is clear that $\phi\iota=1_Q$.
If $t\geq r$, by the composition law we have 
\begin{align*}
  \iota\big((\rho_x,\rho_{x+2z})(\rho_{x'},\rho_{x'+2z'})\big)&=\iota(\rho_{xx'},\rho_{xx'+2(xz'+x'z+2zz')})\\
  &=(xx',0,xz'+x'z+2zz')\\
  &=(x,0,z)\circ (x',0,z'),\\
  &=\iota(\rho_x,\rho_{x+z})\circ \iota(\rho_{x'},\rho_{x'+2z'}).
  \end{align*}
If $t<r$, similar arguments show that $\iota$ is a homomorphism of groups.
Thus $\iota$ is a section of the epimorphism $\phi$.

Let $(x',0,z')=(x,0,z)^{-1}\in\E(C)$. Then the conjugation action of $\iota(Q)$ on $\ker(\phi)$ given by 
\[(x,0,z)(1,1,0)(x',0,z')=\left\{\begin{array}{ll}
  (1,1+2xz',0)&\text{ if }t\geq r\\
  (1,1+2x'z,0)&\text{ if }t<r
\end{array}\right.\]
is consistent with the group action described in the theorem.
\end{proof}

\subsection{Proofs of Theorems \ref{thm:LDU}, \ref{thm:Ecc-H} and Corollary \ref{cor:Ecc:n+2}}\label{sec:ECCsubgrps}

Let $R$ be a ring with identity $1$. Recall that an ideal $I$ of $R$ is \emph{quasi-regular} if $1+I\subseteq U(R)$, where $U(R)$ denotes the set of units of $R$. Idempotents $e_1,\cdots,e_m$ of $R$ are said to be \emph{complete  orthogonal} if $e_1+\cdots+e_m=1$, $e_ie_j=0$ for $i\neq j$. Given a complete orthogonal idempotent $e_1,\cdots,e_m$ of $R$, 
Consider the following subsets of $U(R)$:
\begin{align*}
  L&\coloneqq\{r\in U(R)|e_ire_i=e_i \text{ for all $i$}, \text{ and }e_ire_j=0 \text{ for } i<j\},\\
  D&\coloneqq\{r\in U(R)|e_ire_j=0 \text{ for any } i\neq j\},\\
  U&\coloneqq\{r\in U(R)|e_ire_i=e_i \text{ for all $i$}, \text{ and }e_ire_j=0 \text{ for } i>j\}.
\end{align*} 
Due to Pave{\v{s}}i{\'c} \cite{Pavesic10}, we say that $U(R)$ admits an \emph{LDU-decomposition}, denoted by $U(R)=L\cdot D\cdot U$, if every element of $U(R)$ can be written canonically and uniquely as a product $ldu$ with $l\in L,d\in U,u\in U$.

\begin{lemma}[Theorem 4.13 of \cite{Pavesic10}]\label{lem:LDU}
  If $I$ is a quasi-regular ideal of $R$, then $1+I$ admits an ``LDU"-decomposition with respect to any set of complete orthogonal idempotents.
\end{lemma}

Recall that 
\begin{align*}
  \E_\sharp^k(X)&=\{f\in \E(X)|f_\sharp=id\colon \pi_i(X)\to\pi_i(X),0\leq i\leq k\},\\
  \E_\ast(X)&=\{f\in \E(X)|f_\ast=id\colon H_i(X)\to H_i(X),~i\geq 0\}.
\end{align*} 
In the stable range,  the subsets 
\begin{align*}
  \mathcal{Z}_\sharp^{n+l}(X,Y)&=\{f\in [X,Y]|f_\sharp=0\colon \pi_i(X)\to\pi_i(Y),~\forall i\leq k\},\\
  \mathcal{Z}_\ast(X,Y)&=\{f\in [X,Y]|f_\ast=0\colon H_i(X)\to H_i(Y),~\forall i\geq 0\}
\end{align*}
are subgroups of $[X,Y]$ under addition; the set $[X\vee Y,X\vee Y]$ is a ring with identity $1=1_{X\vee Y}$, and contains the obvious idempotents 
\[e_X=\mat{1_X}{0}{0}{0},\quad e_Y=\mat{0}{0}{0}{1_X}.\]

\begin{proof}[Proof of Theorem \ref{thm:LDU}]
 By induction on $m$, it suffices to prove the theorem in the case $m=2$. 
 Observe that there holds an equality as subgroups of $\E(X\vee Y)$: 
 \[\E_\sharp^{n+l}(X\vee Y)=1+\mathcal{Z}_\sharp^{n+l}(X\vee Y,X\vee Y).\]
Then Lemma \ref{lem:LDU} implies the LDU-decomposition 
\[\E_\sharp^{n+l}(X\vee Y)=\mat{1_X}{0}{\mathcal{Z}_\sharp^{n+l}(X,Y)}{1_Y}\mat{\E_\sharp^{n+l}(X)}{0}{0}{\E_\sharp^{n+l}(Y)}\mat{1_X}{\mathcal{Z}_\sharp^{n+l}(Y,X)}{0}{1_Y}.\]
By \citep[Corollary 3.5]{ALM01}, the group $\E_\sharp^{n+l}(X\vee Y)$ is abelian for $l\geq 2$, and therefore the above LDU-decomposition of $\E_\sharp^{n+l}(X\vee Y)$ is a direct sum decomposition, which completes the proof.
\end{proof}

\begin{proposition}\label{prop:ECC:n+2}
  Let  $C=\CC,C'=C^{n+2,t'}_{r'}$.
  \begin{enumerate}
    \item\label{E1} $\E_\sharp^{n+2}(C)\cong \zz{\min(r,t)}\oplus\zz{}$.
    \item\label{E2} $\mathcal{Z}_\sharp^{n+2}(C,C')\cong \zz{\min(r,t')}\oplus\zz{}$.
    \item\label{E3} $\E_\ast(C)$ has order $2^{\min(r,t)+3}$, and $\mathcal{Z}_\ast(C,C')$ have order $2^{\min(r,t')+3}$.
  \end{enumerate}
  \begin{proof}
  (1) Write a self-map $f$ of $C$ coordinately as \[f=(x,y,z)=x\cdot 1_C+y\cdot i_{n+1}q_{n+1}+z\cdot \omega^t_r,\]
  where $x\in\zz{\max(r,t)+1},y\in\zz{\min(r,t)+1},z\in\zz{\min(r,t)}$.
  Then by Lemma \ref{lem:ECC-eqs} we compute that $(x,y,z)\in\E_\sharp^{n+2}(C)$ if and only if 
 \begin{align*}
  t\geq r:\quad & x=1,y=2u,z=2^{r-1}\epsilon, u\in\zz{r},\epsilon=0,1;\\
  t<r:\quad & x=1+2^{r}\varepsilon,y=2v,z=0, v\in\zz{t},\varepsilon=0,1.
 \end{align*}
It follows that $\E_\sharp^{n+2}(\CC)\cong \zz{\min(r,t)}\oplus\zz{}$, which is generated by $(1,2,0),(1,0,2^{r-1})$ if $t\geq r$; otherwise by $(1,2,0),(1+2^r,0,0)$.

(2) By the group $[C,C']$ and its generators, we divide the discussion into three cases. Let $m=\min(r,t')$. Utilizing relation formulas given by Theorem \ref{thm:eqs}, the following arguments can be carefully verified, the details are omitted here.
\begin{enumerate}[$(i)$]
  \item If $t'\geq t\geq r'\leq r$,  write a map $f\colon C\to C'$ by 
  \[f=(x,y,z)=x\cdot  L(\chi)+y\cdot i_{n+1}q_{n+1}+z\cdot i_{M'}B(\chi^{r+1}_{r'})\bar{\xi}_rq_{\uc},\]
  where $x\in\zz{t+1},y\in\zz{m},z\in\zz{r'}$. There hold formulas:
  \begin{align*}
    (x,y,z)(i_n)&=(x+2z)\cdot i_n,\\
    (x,y,z)(i_{n+1})&=2^{t'-t}x\cdot i_{n+1},\\
    (x,y,z)(i_MB(\chi^1_{r'})\tilde{\eta})&=2^{r-r'}x\cdot i_{M'}B(\chi^1_{r'})\tilde{\eta}+y\cdot i_{n+1}q_{n+1}.
  \end{align*}
It follows that $f\in \mathcal{Z}_\sharp^{n+2}(C,C')$ if and only if \[x=0;\quad y\in \langle 2\rangle \subseteq \zz{m+1};\quad z=2^{r'-1}\epsilon,\epsilon=0,1.\]

  \item If $t'\geq t<r'\leq r$, write a map $f\colon C\to C'$ by 
  \[f=(x,y,z)=x\cdot L(\chi)+y\cdot i_{n+1}q_{n+1}+z\cdot i_{\oc'}\tilde{\xi}_{t'}B(\chi^t_{r'+1})q_M,\]
  where $x\in\zz{r'+1},y\in\zz{m},z\in\zz{t}$. There hold formulas:
  \begin{align*}
    (x,y,z)(i_n)&=x\cdot i_n,\\
    (x,y,z)(i_{n+1})&=2^{t'-t}(x+2z)\cdot i_{n+1},\\
    (x,y,z)(i_{M'}B(\chi^1_{r'})\tilde{\eta})&=2^{r-r'}x\cdot i_{M'}B(\chi^1_{r'})\tilde{\eta}+y\cdot i_{n+1}q_{n+1}.
  \end{align*}
  It follows that $f\in \mathcal{Z}_\sharp^{n+2}(C,C')$ if and only if \[x=2^{r'}\epsilon,\epsilon=0,1;\quad y\in \langle 2\rangle \subseteq \zz{m+1};\quad z=0.\] 
  
  \item If $t'<t\vee r'>r$, write a map $f\colon C\to C'$ by 
  \[f=(x,y,z)=x\cdot i_{\oc'}\tilde{\xi}_{t'}B(\chi^t_{t'+1})q_M+y\cdot i_{n+1}q_{n+1}+z\cdot i_{M'}B(\chi^{r+1}_{r'})\bar{\xi}_rq_{\uc},\]
  where $x\in \zz{t'+1} ,y\in\zz{m},z\in \zz{r+1}$.
 There hold formulas:
 \begin{align*}
  (x,y,z)(i_n)&= 2^{r'-r}z\cdot i_n,\\
  (x,y,z)(i_{n+1})&= x\cdot i_{n+1},\\
  (x,y,z)(i_{M'}B(\chi^1_{r'})\tilde{\eta})&= z\cdot i_{M'}B(\chi^1_{r'})\tilde{\eta}+ y\cdot i_{n+1}q_{n+1}.
 \end{align*}
 It follows that $f\in \mathcal{Z}_\sharp^{n+2}(C,C')$ if and only if \[x=0;\quad y\in \langle 2\rangle \subseteq \zz{m+1};\quad z=2^{r}\epsilon,\epsilon=0,1.\] 
\end{enumerate}
Therefore we get the isomorphism $\mathcal{Z}_\sharp^{n+2}(C,C')\cong\zz{m} \oplus\zz{}$.

(3) The homology groups of $\CC$ is given by (\ref{eq:homology}). 

(i) By computations in (1), we get $(x,y,z)\in\E_\ast(C)$ if and only if 
\[x=1+2^{\max(r,t)}\varepsilon,\quad y\in\zz{\min(r,t)+1},\quad z=2^{\min(r,t)-1}\epsilon\]
for some $\varepsilon,\epsilon\in\{0,1\}$. Hence $\E_\ast(C)$ has order $2^{\min(r,t)+3}$.

(ii) By computations in (2) we get $(x,y,z)\in\mathcal{Z}_\ast(C,C')$ if and only if
$y\in\zz{\min(r,t')+1}$ and 
\[\left\{\begin{array}{ll}
  x=2^{t}\varepsilon, z=2^{r'-1}\epsilon & \text{if }t'\geq t\geq r'\leq r,\\
  x=2^{r'}\varepsilon,z=2^{t-1}\epsilon &\text{ if }t'\geq t<r'\leq r,\\
  x=2^{t'}\varepsilon,z=2^{r}\epsilon&\text{ if }t'<t\vee r'>r.
\end{array}\right.\]
Thus $\mathcal{Z}_\ast(C,C')$ has order $2^{\min(r,t')+3}$.
  \end{proof}
\end{proposition}

\begin{proof}[Proof of Corollary \ref{cor:Ecc:n+2}]
  A direct consequence of Theorem \ref{thm:LDU} and Proposition \ref{prop:ECC:n+2}. 
\end{proof}

\begin{lemma}\label{lem:Ecc-H}
  Let $X,Y$ be $\an$-complexes with $H_{n+2}(X)=H_{n+2}(Y)=0$.
\begin{enumerate}
  \item $\mathcal{Z}_\sharp^{n+2}(X,Y)\leq \mathcal{Z}_\sharp^{n+1}(X,Y)\leq \mathcal{Z}_\ast(X,Y)$ as subgroups under addition.
  \item $\E_\sharp^{n+2}(X)\trianglelefteq \E_\sharp^{n+1}(X)\trianglelefteq \E_\ast(X)$ as normal subgroups.
\end{enumerate}

\begin{proof}
(1) It suffices to show $\mathcal{Z}_\sharp^{n+1}(X,Y)\subseteq \mathcal{Z}_\ast(X,Y)$. Given a map $f\colon X\to Y$ with $\pi_{n+i}(f)=0$ for $i=0,1$, the naturality of the Hurewicz homomorphisms
\[
  h_n\colon \begin{tikzcd}[sep=small]
    \pi_n(X)\ar[r,"\cong"]&H_n(X)
  \end{tikzcd},\quad h_{n+1}\colon \begin{tikzcd}[sep=small]
    \pi_{n+1}(X)\ar[r,two heads]& H_{n+1}(X)
  \end{tikzcd}
\]
with respect to $f$ implies that $H_{n+1}(f)=0,H_n(f)=0$. Since $H_{n+i}(X)=0$ for $i\geq 2$, we get $f\in \mathcal{Z}_\ast(X,Y)$.

(2) Given a map $f\in \E(X)$	such that $\pi_{n+i}(f)=id$ for $i=0,1$. The same arguments show that 
$H_n(f)$ and $H_{n+1}(f)$ are the identities; that is, $f\in\E_\ast(X)$. Thus we have the inclusion $\E_\sharp^{n+1}(X)\subseteq \E_\ast(X)$, which completes the proof.
\end{proof}	
\end{lemma}

\begin{proof}[Proof of Theorem \ref{thm:Ecc-H}]
Write $C_i=C^{n+2,t_i}_{r_i}$, $i=1,\cdots,m$, and let $X=C_1\vee \cdots\vee C_m$. 
By Lemma \ref{lem:Ecc-H} there is an exact sequence 
\[0\to \E_\sharp^{n+2}(X)\to  \E_\ast(X)\xra{(\pi_{n+1},\pi_{n+2})}G_1\oplus G_2\to 0,\]
where $G_i=\pi_{n+i}(\E_\ast(X))\leq \Aut(\pi_{n+i}(X)), i=1,2$.
By Proposition \ref{prop:ECC:n+2} (\ref{E3}) we get that $\E_\ast(X)$ has order 
\[\prod_{i=1}^m|\E_\ast(C_i)|\cdot \prod_{i\neq j=1}^{m,m}|\mathcal{Z}_\ast(C_i,C_j)|=\prod_{i,j=1}^{m,m}2^{\min(r_i,t_j)+3}.\]
By Corollary \ref{cor:Ecc:n+2} it suffices to show that $G_1\cong G_2\cong (\zz{})^{m^2}$.

Suppose that $f=(f_{ij})\in \E_\ast(X)$. Then 
\[H_{n+1}(f_{ij})\equiv \delta_{ij}\pmod{2^{t_i}}.\]
Note that $H_{n+1}(X)\cong \zz{t_1}\oplus\cdots\zz{t_m}$ with each $\zz{t_i}$ generated by $(i_{n+1})_\ast\sigma_{n+1}$, while $\pi_{n+1}(X)\cong \zz{t_1+1}\oplus\cdots\zz{t_m+1}$ with each $\zz{t_i+1}$ generated by $i_{n+1}$. 
The coherence of generators of $H_{n+1}(X)$ and $\pi_{n+1}(X)$ then implies that 
\[\pi_{n+1}(f_{ij})=\delta_{ij}+2^{t_i}\varepsilon_{ij},~~\varepsilon_{ij}\in\{0,1\}.\]
 For another map $f'=(f_{ij}')\in\E_\ast(X)$, express $\pi_{n+1}(f')$ similarly, then the $(i,j)$-entry of matrix product $\pi_{n+1}(f_{ij})\cdot\pi_{n+1}(f_{ij}')$ is of the form 
\begin{align*}
  \sum_k [(\delta_{ik}+2^{t_i}\varepsilon_{ik})(\delta_{kj}+2^{t_k}\varepsilon_{kj})]&=\sum_k (\delta_{ik}\delta_{kj}+\delta_{ik}2^{t_k}\varepsilon_{kj}'+\delta_{kj}2^{t_i}\varepsilon_{ik})\\
  &=\delta_{ij}+2^{t_i}(\varepsilon_{ij}+\varepsilon_{ij}').
\end{align*}
It follows that $G_1$ is commutative, and therefore $G_1\cong (\zz{})^{m^2}$.

Recall that $\pi_{n+2}(C_1\vee \cdots\vee C_m)\cong (\zz{}\oplus\zz{})^{m^2}$, each direct summand $\zz{}\oplus\zz{}$ has a generating set $\{i_{n+1}\eta,i_MB(\chi)\tilde{\eta}\}$. From the computations in Proposition \ref{prop:ECC:n+2} we see that only the generator $i_{n+1}q_{n+1}$ of each group $[C_j,C_i]$ has effect on $\pi_{n+2}(C_1\vee \cdots\vee C_m)$. We have formulas
\[i_{n+1}q_{n+1}\circ (i_{n+1}\eta)=0,\quad i_{n+1}q_{n+1}\circ (i_MB(\chi)\tilde{\eta})=i_{n+1}\eta.\]
It follows that  
\[\pi_{n+2}(f_{ij})=\begin{pmatrix}
  \delta_{ij}&\epsilon_{ij} \\
  0&\delta_{ij}
\end{pmatrix}\in \mathrm{End}\big(\langle i_{n+1}\eta\rangle \oplus \langle i_MB(\chi)\tilde{\eta}\rangle\big) \]
for some $\epsilon_{ij}\in\{0,1\}$.
For another map $f'=(f_{ij}')\in\E_\ast(X)$, express $\pi_{n+2}(f')$ similarly, then the $(i,j)$-entry (block) of the matrix $\pi_{n+2}(f_{ij})\cdot\pi_{n+2}(f_{ij}')$ is of the form 
\[
  \sum_k\begin{pmatrix}
    \delta_{ik}&\epsilon_{ik} \\
    0&\delta_{ik}
  \end{pmatrix}\begin{pmatrix}
    \delta_{kj}&\epsilon_{kj}' \\
    0&\delta_{kj}
  \end{pmatrix}=\begin{pmatrix}
    \delta_{ij}&\epsilon_{ij}+\epsilon_{ij}' \\
    0&\delta_{ij}
  \end{pmatrix}.
\]
Thus $G_2$ is also commutative, and therefore $G_2\cong (\zz{})^{m^2}$.
\end{proof}

\bibliographystyle{amsplain}
\bibliography{Refs}

\end{document}